\documentclass[12pt,a4paper]{amsart}%
\usepackage{amssymb}
\usepackage{amsmath}
\usepackage{amsfonts}
\usepackage{doublespace}
\usepackage{color}
\usepackage{graphicx}%
\providecommand{\U}[1]{\protect\rule{.1in}{.1in}}
\theoremstyle{plain}
\newtheorem{theorem}{Theorem}

\newtheorem{proposition}[theorem]{Proposition}

\numberwithin{equation}{section}
\begin{document}
\title[Dirichlet-Neumann problem]{Inverse Dirichlet to Neumann problem for nodal curves}
\author{Gennadi Henkin}
\address{
\begin{spacing}{1.1}%
[G. Henkin and V. Michel]Universit\'{e} Pierre et Marie Curie\\
4, place Jussieu, 75252 Paris Cedex 05, France\\
[G. Henkin] CEMI, Academy of Science, 117418, Moscow, Russia%
\end{spacing}%
}
\email{
\begin{spacing}{1.1}%
henkin@math.jussieu.fr, michel@math.jussieu.fr%
\end{spacing}%
}
\author{Vincent Michel}
\date{31/10/2012}

\begin{abstract}%
\begin{spacing}{1.1}%
This paper proposes direct and inverse results for the Dirichlet and Dirichlet
to Neumann problems\ for complex curves with nodal type singularities. As an
application, we give a method to reconstruct the conformal structure of a
compact surface of $\mathbb{R}^{3}$ with constant scalar conductivity from
electrical current measurements in a neighborhood of one of its points.%
\end{spacing}%

\end{abstract}
\maketitle
\tableofcontents

%

\begin{spacing}{1.1}%

\section{Introduction}

Let $Z$ be a compact or open bordered surface of $\mathbb{R}^{3}$ equipped
with the complex structure induced by the standard euclidean metric of
$\mathbb{R}^{3}$~; this point of view on Riemann surface, which goes to a
result of Gauss about isothermal coordinates in 1822, is not restrictive since
it has been proved by Garsia~\cite{GaA1961} for the compact case and by
R\"{u}edy~\cite{RuR1971} for the bordered case that any abstract Riemann
surface is isomorphic to such a manifold. Let $\overline{\partial}$ be the
Cauchy-Riemann operator of $Z$, $d^{c}=i\left(  \overline{\partial}%
-\partial\right)  $ and $d=\partial+\overline{\partial}$. If $Z$ has a
constant scalar conductivity and if there is no time fluctuation and no source
nor sink of current, it follows from the Maxwell's equations that an
electrical potential on an open set of $Z$ is a smooth function $U$ which
satisfies the equation $dd^{c}U=0$~; the form $d^{c}U=i\left(  \overline
{\partial}U-\partial U\right)  $ can be then seen as modeling the physical
current arising from the potential $U$ (see e.g.\ \cite{SyJ1990}). An isolated
finite charge induces a current with a simple pole. When the current $d^{c}U$
is theoretically allowed to have singularities on a discrete set, it is
natural to limit them to simple charged poles. The fact that charges should
somehow compensate and arise from simple poles is mathematically natural
because\ according to proposition~\ref{SingLog}, singular potentials can be
seen as harmonic distributions on a complex nodal curve. The theorem below
gives an electrostatic interpretation for an accurate Dirichlet problem when a
discrete set of finites charges is allowed. \smallskip

\noindent\textbf{Theorem} (Riemann 1851, Klein 1882).\textbf{\ }\textit{Let
}$Z$\textit{\ be a compact or bordered connected oriented smooth surfaces in
}$\mathbb{R}^{3}$\textit{\ equipped with the conformal structure induced by
the standard euclidean metric of }$\mathbb{R}^{3}$\textit{. Let }%
$\overline{\partial}$ \textit{be its} \textit{Cauchy-Riemann operator, }%
$d^{c}=i\left(  \overline{\partial}-\partial\right)  $\textit{\ and
}$d=\partial+\overline{\partial}$. \textit{Assume }$u$\textit{\ is an
electrical potential on }$bZ$\textit{\ (this assumption is empty when }%
$Z$\textit{\ is compact) and that }$Z$\textit{\ has electrical real charges
}$\pm c_{j}$\textit{\ concentrated at points }$a_{j}^{\pm}$, $1\leqslant
j\leqslant\nu$\textit{. Then there is a uniqque electrical potential }$U$
\textit{extending} \textit{(when }$Z$\textit{\ is non compact) }$u$
\textit{to} $Z$ \textit{such that }$dd^{c}U=0$\textit{\ on }$Z\backslash
\left\{  a_{j}^{\pm};~1\leqslant j\leqslant\nu\right\}  $\textit{\ and the
residue }$\operatorname*{Res}_{a_{j}^{\pm}}\left(  d^{c}U\right)
\overset{def}{=}\frac{1}{2\pi}\int_{dist(.,a_{j}^{\pm})=\varepsilon}d^{c}U$
($\varepsilon>0$ small enough) \textit{of }$d^{c}U=i\left(  \overline
{\partial}-\partial\right)  U$\textit{\ at }$a_{j}^{\pm}$\textit{\ is }$\pm
c_{j}$, $1\leqslant j\leqslant\nu$.\smallskip\textit{\ }

This problem was firstly considered by Gauss in 1840, Tomson (also named Lord
Kelvin) and Dirichlet in 1847. Riemann gave in 1851 a mathematically
incomplete proof. Klein wrote in 1882 an electrostatic interpretation which
has been considered as a sufficient justification by physicists. Effective and
correct constructions were given by Fredholm in 1899 and Hilbert in 1901. A
good report of this story can be found in a book of de
Saint-Gervais~\cite{SG2010Li}.

In 1962, Gelfand~\cite{GeI1962} formulated and obtained the first non trivial
result in the inverse problem of reconstructing the complex structure of a
compact surface in $\mathbb{R}^{3}$ from the knowledge of the spectrum of its
laplacien. This problem has been solved for most surfaces by
Buser~\cite{BuP1997} in 1997. A related inverse question has been enunciated
by Wentworth in 2010~: How to recover the conformal structure of a compact
Riemann surface from Dirichlet to Neumann data type of some subdomain. The
theorem~\ref{T/ compactR3} below, which is a development of inverse results
contained in \cite[section~2]{HeG-MiV2007}, is an inverse version for the
compact case of the Riemann and Klein theorem and gives a constructive answer
to Wentworth's question.

Before we formulate theorem~\ref{T/ compactR3opt}, we set up some definitions
and notations. Let $Z$\ be a compact connected oriented smooth surface in
$\mathbb{R}^{3}$\ equipped with the conformal structure induced by the
standard euclidean metric of $\mathbb{R}^{3}$. We denote by $D_{Z}$ the set of
couples $\left(  a,c\right)  $ in $Z^{6}\times\mathbb{R}^{3}$ such that
$a=\left(  a_{\ell}^{-},a_{\ell}^{+}\right)  _{0\leqslant\ell\leqslant2}$ is a
family of six mutually distinct points and $c=\left(  c_{\ell}\right)
_{0\leqslant\ell\leqslant2}\in\mathbb{R}^{3}$. If $\left(  a,c\right)  \in
D_{Z}$, $U_{Z,\ell}^{a,c}$ denotes a function which is harmonic on
$Z\backslash\left\{  a_{\ell}^{-},a_{\ell}^{+}\right\}  $ and such that
$\partial U_{Z,\ell}^{a,c}$ has a simple pole at $a_{\ell}^{\pm}$ with residue
$\pm c_{\ell}$~; as a matter of fact, $U_{Z,\ell}^{a,c}$ is a standard Green
bipolar function and while it is determined only up to an additive constant,
$\partial U_{Z,\ell}^{a,c}$ is unique. Then, as explained in section~3, if the
$\partial U_{Z,\ell}^{a,c}$ have no common zero, we can define a map
$F_{Z}^{a,c}=\left(  \partial U_{Z,0}^{a,c}:\partial U_{Z,1}^{a,c}:\partial
U_{Z,0}^{a,c}\right)  $ from $Z\backslash\left\{  a_{\ell}^{\pm}%
;~0\leqslant\ell\leqslant2\right\}  $ to $\mathbb{CP}_{2}$ which in $\left\{
\partial U_{Z,0}^{a,c}\neq0\right\}  $ has the affine representation $\left(
\frac{\partial U_{Z,1}^{a,c}}{\partial U_{Z,0}^{a,c}},\frac{\partial
U_{Z,2}^{a,c}}{\partial U_{Z,0}^{a,c}}\right)  $~; the quotients here are well
defined meromorphic functions because $\dim Z=1$. We denote by $E_{Z}$ the set
of $\left(  a,c\right)  $ in $D_{Z}$ such that $F_{Z}^{a,c}$ is well defined
and injective outside some finite subset of $Z$ . It is clear that $E_{Z}$ is
an open subset of $D_{Z}$. We can now give an inverse result for the compact
case of the Riemann and Klein theorem.

\begin{theorem}
\label{T/ compactR3opt}Let $Z$\ and $Z^{\prime}$ be compact connected oriented
smooth surfaces in $\mathbb{R}^{3}$\ equipped with the conformal structures
induced by the standard euclidean metric of $\mathbb{R}^{3}$. Assume that
$Z\cap Z^{\prime}$ contains a surface $S$ and let $a=\left(  a_{\ell}%
^{-},a_{\ell}^{+}\right)  _{0\leqslant\ell\leqslant2}$ be a 6-uple of mutually
distinct points in $S$. Assume that for some $c\in\mathbb{R}^{3}$, $\left(
a,c\right)  \in E_{Z}\cap E_{Z^{\prime}}$ and $\left(  U_{Z,\ell}%
^{a,c}\left\vert _{S}\right.  \right)  _{0\leqslant\ell\leqslant2}=\left(
U_{Z^{\prime},\ell}^{a,c}\left\vert _{S}\right.  \right)  _{0\leqslant
\ell\leqslant2}$. Then $Z$ and $Z^{\prime}$ are isomorphic. Moreover, $Z$ can
be explicitly reconstructed from $\left(  \partial U_{Z,\ell}^{a,c}\left\vert
_{S}\right.  \right)  _{0\leqslant\ell\leqslant2}$.
\end{theorem}

The practical interest of this result would be greatly improved if $E_{Z}$
would be dense in $Z^{6}\times\mathbb{C}^{3}$. If this seems very likely, it
has yet to be proved. So we slightly modify our point of view in allowing
small perturbations. Let us be precise. For $n\in\mathbb{N}^{\ast}$, denote by
$D_{Z,n}$ the set of 4-uples $\left(  a,c,p,\kappa\right)  $ in $Z^{6}%
\times\mathbb{C}^{3}\times\left(  Z^{n}\right)  ^{3}\times\left(
\mathbb{C}^{n}\right)  ^{3}$ such that $\left(  a,c\right)  \in D_{Z}$,
$p=\left(  p_{\ell}\right)  _{0\leqslant\ell\leqslant2}$ and $\kappa=\left(
\kappa_{\ell}\right)  _{0\leqslant\ell\leqslant2}$ where each $p_{\ell
}=\left(  p_{\ell,j}\right)  _{1\leqslant j\leqslant n}$ is a family of
mutually distinct points of $Z\backslash\left\{  a_{0}^{-},a_{0}^{+},a_{1}%
^{-},a_{1}^{+},a_{2}^{-},a_{2}^{+}\right\}  $ and each $\kappa_{\ell}=\left(
\kappa_{\ell,j}\right)  _{1\leqslant j\leqslant n}\in\mathbb{C}^{n}$ satisfies
$%
{\displaystyle\sum\limits_{1\leqslant j\leqslant n}}
\kappa_{\ell,j}=0$. If $\left(  a,c,p,\kappa\right)  \in E_{Z,n}$, we denote
by $V_{Z,\ell}^{p,\kappa}$ a function which is harmonic on $Z\backslash
\left\{  p_{\ell,j};~1\leqslant j\leqslant n\right\}  $ and such that
$\partial V_{Z,\ell}^{p,\kappa}$ has simple poles at $p_{\ell,j}$ with
residues $\kappa_{\ell,j}$~;$\ V_{Z}^{p,\kappa}$ is unique up to an additive constant.

We denote by $E_{Z,n}$ the set of $\left(  a,c,p,\kappa\right)  \in D_{Z,n}$
such that $F_{Z}^{a,c,p,\kappa}=\left(  \partial U_{Z,0}^{a,c}+\partial
V_{Z,0}^{p,\kappa}:\partial U_{Z,1}^{a,c}+\partial V_{Z,1}^{p,\kappa}:\partial
U_{Z,2}^{a,c}\right)  $ is well defined and injective outside some finite
subset of $Z$. It is clear that $E_{Z,n}$ is open in $D_{Z,n}$. According to
proposition~\ref{P/ compact generique}, elements of $E_{Z,n}$ can be called
\textit{generic}.

\begin{proposition}
\label{P/ compact generique}Let $Z$\ be a compact connected oriented smooth
surface in $\mathbb{R}^{3}$\ equipped with the conformal structure induced by
the standard euclidean metric of $\mathbb{R}^{3}$ and $S$ a subdomain of $Z$.
Consider $\left(  a,c\right)  $ in $D_{Z}$ with $a\in S^{6}$. Then there is
$n\in\mathbb{N}^{\ast}$ such that for any neighborhood $W$ of $0$ in
$\mathbb{C}$, there exists $\left(  p,\kappa\right)  \in\left(  S^{n}\right)
^{3}\times\left(  W^{n}\right)  ^{3}$ such that $\left(  a,c,p,\kappa\right)
\in E_{Z,n}$.
\end{proposition}

This result as well as proposition~\ref{P/ generique B} will be proved in a
separate paper because they involve methods and results of complex analysis
which deserve attention of their own. We can now give a generic version of
theorem~\ref{T/ compactR3opt}.

\begin{theorem}
\label{T/ compactR3}Let $Z$\ and $Z^{\prime}$ be a compact connected oriented
smooth surface in $\mathbb{R}^{3}$\ equipped with the conformal structures
induced by the standard euclidean metric of $\mathbb{R}^{3}$. Assume that
$Z\cap Z^{\prime}$ contains a surface $S$ and let $a=\left(  a_{\ell}%
^{-},a_{\ell}^{+}\right)  _{0\leqslant\ell\leqslant2}$ be a 6-uple of mutually
distinct points in $S$. Assume that for some $\left(  c,p,\kappa\right)
\in\mathbb{C}^{3}\times\left(  S^{n}\right)  ^{3}\times\left(  \mathbb{C}%
^{n}\right)  ^{3}$, $\left(  a,c,p,\kappa\right)  \in E_{Z,n}\cap
E_{Z^{\prime},n}$ and $\left(  U_{Z,\ell}^{a,c,p,\kappa}\left\vert
_{S}\right.  \right)  _{0\leqslant\ell\leqslant2}=\left(  U_{Z^{\prime},\ell
}^{a,c,p,\kappa}\left\vert _{S}\right.  \right)  _{0\leqslant\ell\leqslant2}$.
Then $Z$ and $Z^{\prime}$ are isomorphic. Moreover, $Z$ can be explicitly
reconstructed from $\left(  \partial U_{Z,\ell}^{a,c,p,\kappa}\left\vert
_{S}\right.  \right)  _{0\leqslant\ell\leqslant2}$.
\end{theorem}

\noindent\textbf{Remarks. 1. }As a consequence, the genus of $Z$ is fully
determined by the data $S$ and $\left(  U_{Z,\ell}^{a,c,p,\kappa}\left\vert
_{S}\right.  \right)  _{0\leqslant\ell\leqslant2}$ here considered but,
meanwhile, a formula has yet to be found.

\textbf{2. }The result apply for compact (generalized) nodal curves if the
potentials are associated to generic admissible families of $Z$ and
$Z^{\prime}$ (see sections~2 and~3 for definitions)\smallskip

Our next theorem is an inverse version for the bordered case of the Riemann
and Klein result. Without electrical charges, it is contained in \cite[th.~1,
th.2]{HeG-MiV2007}. The precise definition Dirichlet-Neumann data and how they
are linked to the Neumann operator is given in the third section. In short,
such a datum consists of a smooth oriented real curve $\gamma$ which is the
boundary of an open complex curve $Z$, a 3-uple $u=\left(  u_{0},u_{1}%
,u_{2}\right)  $ of smooth real functions defined on $\gamma$, a 3-uple
$\theta u=\left(  \theta u_{0},\theta u_{1},\theta u_{2}\right)  $ of smooth
$\left(  1,0\right)  $-forms, each $\theta u_{\ell}$ being the boundary value
of $\partial\widetilde{u_{\ell}}^{c}$ where $\widetilde{u_{\ell}}^{c}$ is the
harmonic extension of $u_{\ell}$ to $Z\backslash\left\{  \left(  a_{j}%
^{-},a_{j}^{+}\right)  ;~1\leqslant j\leqslant\nu\right\}  $ such that
$\partial\widetilde{u_{\ell}}^{c}$ has residue $\pm c_{j}$ at $a_{j}^{\pm}$,
$1\leqslant j\leqslant\nu$. An important but, as seen later, generic (see
proposition~\ref{P/ generique B}), requirement for $\ \left(  \gamma,u,\theta
u\right)  $ to be a Dirichlet-Neumann datum is that the map $\left(
\partial\widetilde{u_{0}}^{c}:\partial\widetilde{u_{1}}^{c}:\partial
\widetilde{u_{2}}^{c}\right)  $ is well defined and injective outside a finite
subset of $Z$ and embeds $\gamma$ into $\mathbb{CP}_{2}$. For the sake of
simplicity and because charges are informations to be recovered, we have
chosen to let them be independent of $\ell$. \smallskip

\begin{theorem}
\label{T/ bordR3}Let $Z$\ and $Z^{\prime}$ be bordered oriented smooth
surfaces in $\mathbb{R}^{3}$\ equipped with the conformal structures induced
by the standard euclidean metric of $\mathbb{R}^{3}$. We fix in $Z$ (resp.
$Z^{\prime}$) $\nu$ (resp. $\nu^{\prime}$) pairs of mutually distinct points
$a_{j}^{\pm}$ (resp. $a_{j}^{\prime\pm}$) in $Z$ (resp. $Z^{\prime}$). We
assign to each pair $a_{j}^{\pm}$ (resp. $a_{j}^{\prime\pm}$) "electrical" non
zero complex charges $\pm c_{j}$ (resp. $\pm c_{j}^{\prime} $) satisfying the
generic conditions $\left\vert c_{j}^{\pm}\right\vert \neq\left\vert
c_{k}^{\pm}\right\vert $ (resp. $\left\vert c_{j}^{\prime\pm}\right\vert
\neq\left\vert c_{k}^{\prime\pm}\right\vert $) $1\leqslant j<k\leqslant\nu$.

We assume that $\left(  \gamma,u,\theta u\right)  $ is a Dirichlet-Neumann
datum for $Z$ where each $a_{j}^{\pm}$ is charged with $c_{j}^{\pm}$ as well
as a Dirichlet-Neumann datum for $Z^{\prime}$ where each $a_{j}^{\prime\pm}$
is charged with $c_{j}^{\pm}$.

Then $\nu=\nu^{\prime}$, $\left(  c_{j}\right)  _{1\leqslant j\leqslant\nu
}=\left(  c_{j}^{\prime}\right)  _{1\leqslant j\leqslant\nu}$ and there is an
isomorphism $\varphi:Z\longrightarrow Z^{\prime}$ of Riemann surfaces such
that $\varphi\left\vert _{\gamma}\right.  =Id_{\gamma}$ and $\varphi\left(
a_{j}^{\pm}\right)  =a_{j}^{\prime\pm}$, $1\leqslant j\leqslant\nu$. Moreover,
$Z$, $\left\{  a_{j}^{\pm};~1\leqslant j\leqslant\nu\right\}  $ and $\left(
c_{j}\right)  _{1\leqslant j\leqslant\nu}$ can be explicitly reconstructed
from $\left(  \gamma,u,\theta u\right)  $.
\end{theorem}

The proof of the theorem~\ref{T/ compactR3} is given in section~3.2. Because
the pairs $\left\{  a_{j}^{-},a_{j}^{+}\right\}  $ can be seen as the
singularities of a nodal curve, theorem~\ref{T/ bordR3} is a consequence of
theorems~\ref{T/ unicite} and~\ref{T/ looseIDN} which deal with (generalized)
nodal curves. Its proof is given at the end of section~3 where are stated our
main theorems about inverse problems. Section~2 is devoted to the definition
of nodal surfaces, harmonic distributions and to Dirichlet problems. A
characterization of nodal curve Dirichlet-Neumann data is given in section~4.

\section{Direct problems for compact and nodal curves}

\subsection{Nodal curves}

In this paper, an \textit{open bordered Riemann surface} is the interior of a
one dimensional compact complex manifold with boundary whose all connected
components have non trivial one real dimensional smooth boundary.
An\textit{\ open bordered nodal curve }$X$ is the quotient of an open
bordered\textit{\ }Riemann surface $Z$ by an equivalence relation identifying
a finite number of interior points$~$; $Z$ is said to be \textit{above} $X$.
We define likewise \textit{compact nodal curve }but for that case we require
the connectedness of the compact Riemann surface.

The points of the singular set $\operatorname*{Sing}X$ of a nodal curve $X$
are called nodes. As a consequence, the irreducible components of $X$ at one
of its node are germs of Riemann surfaces. If $a$ is any point of $X$, we call
the \textit{branches} of $X$ at $a$ any family $\left(  X_{a,j}\right)
_{1\leqslant j\leqslant\nu\left(  a\right)  }$ of connected Riemann surfaces
meeting only at $a$ and whose union is a relatively compact neighborhood of
$a$ in $X$. Thus, the nodes of $X$ are the points $a$ of $X$ where $\nu\left(
a\right)  \geqslant2$. When $\nu=2$ at each node of $X$, a restriction not
relevant for our theorems, $X$ is a nodal curve\textit{\ }as often defined in
the literature. That's why we omit most of the time to add the word
\textit{generalized }for the nodal curves we consider. When $X$ is compact,
has a trivial group of automorphisms and $\nu\equiv2$, $X$ is called
\textit{stable} (see e.g. \cite{Ha-Mo1998Li}).

In the sequel $X_{a,j}^{\ast}$ is a notation for $X_{a,j}\backslash\left\{
a\right\}  $. The boundary of $X$ is denoted by $bX$; by definition
$\overline{X}=X\cup bX$ is outside $\operatorname*{Sing}X$ a manifold with
boundary~; its regular part $\operatorname*{Reg}\overline{X}$ is
$bX\cup\operatorname*{Reg}X$ where $\operatorname*{Reg}X=X\backslash
\operatorname*{Sing}X$. Note that when $X$ is an open bordered nodal curve and
$\overline{\widehat{X}}\overset{\pi}{\longrightarrow}\overline{X}$
($\widehat{X}\overset{\pi}{\longrightarrow}X $ for short but with a slight
abuse of notation) is one of its normalization, $\widehat{X}$ is an open
bordered Riemann surface and $\pi^{-1}\left(  \operatorname*{Sing}X\right)  $
is finite.

Two open bordered nodal curves $X$ and $X^{\prime}$ are said
\textit{isomorphic} if there exists a bijective map $\varphi:\overline
{X}\longrightarrow\overline{X^{\prime}}$ which is an isomorphism of Riemann
surfaces from $\operatorname*{Reg}X$ onto $\operatorname*{Reg}X^{\prime}$, a
diffeomorphism of manifolds with boundary between some open neighborhoods of
$bX$ and $bX^{\prime}$ in $\overline{X}$ and $\overline{X^{\prime}}$ and such
that for each node $a$ of $X$, the (germs of) branches of $X^{\prime}$ at
$\varphi\left(  a\right)  $ are the images by $\varphi$ of the (germs of)
branches of $X$ at $a$. In particular, such a map $\varphi$ has to be an homeomorphism.

A weaker notion of equivalence between nodal curves appears naturally in this
paper. If the above map $\varphi$ has the first two properties but only send
bijectively the set (of germs) of branches of $X$ to the set of (of germs) of
branches of $X^{\prime}$, we say that $X$ and $X^{\prime}$ are
\textit{roughly} \textit{isomorphic.} Assuming $X$ (resp. $X^{\prime}$) is the
quotient of an open bordered Riemann surface $Z$ (resp. $Z^{\prime}$) and that
$Z\overset{\pi}{\longrightarrow}X$ (resp. $Z^{\prime}\overset{\pi
}{\longrightarrow}X^{\prime})$ is the natural projection, another way to state
this is to ask for an isomorphism of open bordered Riemann surface from
$\overline{Z}$ onto $\overline{Z^{\prime}}$ which sends $\pi^{-1}\left(
\operatorname*{Sing}X\right)  $ onto $\pi^{\prime-1}\left(
\operatorname*{Sing}X^{\prime}\right)  $.

In order to settle accurate Dirichlet problems, we define in the next
subsection what is a harmonic distribution on a nodal curve.

\subsection{Harmonic distributions}

When a nodal curve $X$ is an analytic subset of some open set in an affine
space, one may agree to define smooth functions as restrictions to $X$ of
smooth functions of the ambient space. In the particular simple case where $X
$ is the union of the lines $\mathbb{C}\left(  1,0\right)  $ and
$\mathbb{C}\left(  0,1\right)  $, it appears that a function is smooth on $X$
if and only if it is smooth on each branch of $X$ and is continuous at the
singular point of $X$. Having in mind that every open bordered nodal curve can
be embedded in an affine complex space (see \cite{WiK1966}), we take this
model as a guideline for a general definition. Translating in analytic words
the algebraic definitions of \cite{RoM1954}, \cite{GrA1958} and \cite{SeJ1959}
would have given the same result.

Let $X$ be a (generalized) nodal curve, $W$ an open set of $X$ and
$r\in\left[  0,+\infty\right]  $. A function $u$ on $W$ is said to be of class
$C^{r}$ if it is continuous and if for any branch $B$ of $X$ contained in $W$,
$u\left\vert _{B}\right.  \in C^{r}\left(  B\right)  $~; $\ $the space of such
functions is denoted by $C^{r}\left(  W\right)  $ or $C_{0,0}^{r}\left(
W\right)  $.

If $p,q\in\left\{  0,1\right\}  $ and $p+q>0$, a $\left(  p,q\right)  $-form
$\omega$ of $C_{p,q}^{r}\left(  W\cap\operatorname*{Reg}X\right)  $ is said to
be of class $C^{r}$ on $W$ if for any branch $B$ of $X$ contained in $W$,
$\omega\left\vert _{B\cap\operatorname*{Reg}X}\right.  $ extends as an element
of $C_{p,q}^{r}\left(  B\right)  $. The space of such forms is denoted by
$C_{p,q}^{r}\left(  W\right)  $. Note that the question of continuity at nodes
of a form is relevant only when it is a function since the tangent spaces of
branches of $X$ at a\ same node may be different.

If $K$ is a compact subset of $X$ and $p,q\in\left\{  0,1\right\}  $, the
space $C_{p,q}^{\infty}\left(  K\right)  $ of smooth $\left(  p,q\right)
$-forms supported in $K$ is equipped with the topology induced by the
semi-norms $\underset{_{B\cap K}}{\sup}\left\Vert D^{\left(  m\right)  }%
\omega\left\vert _{B}\right.  \right\Vert $ where $m$ is any positive integer,
$B$ any branch of $X$ meeting $K$ and the differential $D$ is the total
differential acting on coefficients. The space $D_{p,q}\left(  W\right)  $ of
smooth $\left(  p,q\right)  $-forms compactly supported in $W$ is equipped
with the inductive limit topology of the spaces $C_{p,q}^{\infty}\left(
K\right)  $ where $K$ is any compact of $W$. The space $D_{p,q}^{\prime
}\left(  W\right)  $ of currents on $W$ of bidegree $\left(  p,q\right)  $ is
the topological dual of $D_{p,q}\left(  W\right)  ~$; the elements of
$D_{1,1}^{\prime}\left(  W\right)  $ are the distributions on $W$.

The exterior differentiation $d$ of smooths forms is well defined along
branches of $X$, so it is for $\partial$ and $\overline{\partial}$. These
operators extend to currents by duality.

A distribution $u\in D_{1,1}^{\prime}\left(  W\right)  $ is \textit{(weakly)
harmonic} if, by definition, the current $i\partial\overline{\partial}u$
vanish, that is $\left\langle i\partial\overline{\partial}u,\varphi
\right\rangle =0$ for all $\varphi\in C_{c}^{\infty}\left(  W\right)  $; it is
equivalent to ask for $\partial u$ to be (weakly) holomorphic in the sense of
Rosenlicht~\cite{RoM1954}. Such distributions are usual harmonic functions
near regular points. While an harmonic function on $\operatorname*{Reg}X$ may
have heavy singularities at a node, the following proposition shows that
harmonic distributions have at most logarithmic singularities which somehow
compensates together.

\begin{proposition}
[Characterization of harmonic distributions]\label{SingLog}Let $X$ be a nodal
curve. Assume that $u$ is a harmonic distribution on a neighborhood of some
point $a$ in $X$. Let $\left(  X_{a,j}\right)  _{1\leqslant j\leqslant
\nu\left(  a\right)  }$ be the branches of $X$ at $a$. Assume that these
branches are small enough for there exists on each $X_{a,j}$ a holomorphic
coordinate $z_{j}$ centered at $a$. Then, there exists a family $\left(
c_{a,j}\right)  _{1\leqslant j\leqslant\nu\left(  a\right)  }$ of complex
numbers such that $u\left\vert _{X_{a,j}^{\ast}}\right.  -2c_{a,j}%
\ln\left\vert z_{j}\right\vert $ extends as a usual harmonic function near $a$
in $X_{a,j}$ and $%
{\displaystyle\sum\limits_{1\leqslant j\leqslant\nu\left(  a\right)  }}
c_{a,j}=0$. In particular, $\partial u$ is a meromorphic (1,0)-form whose
singularities are simple poles at nodes of $X$ and has residue $c_{a,j}$ along
$X_{a,j}$ when $a\in\operatorname*{Sing}X$. Conversely, $u$ is a harmonic
distribution if $\partial u$ and $\left(  c_{a,j}\right)  $ are such.
\end{proposition}

\noindent\textbf{Remark. }The condition that the singularities of $\partial u$
are only simple poles with vanishing sums of residues at each node
characterizes that $\partial u$ is (weakly) holomorphic. This fact match the
definition of dualizing sheaves given by Grothendieck in~\cite{GrA1958} and
Hartshone in~\cite{HaR166Li} which is an algebraic point of view for (weakly)
holomorphic forms.

\begin{proof}
[Proof]For each $k$, we assume that $X_{a,k}$ is small enough so that $z_{k}$
is bijective from $X_{a,k}$ onto $\mathbb{D}=D\left(  0,1\right)  $ and we fix
some $\xi_{k}\in C_{c}^{\infty}\left(  X_{a,k}\right)  $ such that $\xi_{k}=1$
in a neighborhood of $a$ in $X_{a,k}$. Consider some $j$ in $\left\{
1,..,\nu\left(  a\right)  \right\}  $. The extension operator which to
$\chi\in D_{1,1}\left(  X_{a,j}\right)  $ associates the form $E_{j}\chi$
defined by $\left(  E_{j}\chi\right)  \left\vert _{X_{a,k}}\right.  =0$ if
$k\neq j$ and $\left(  E_{j}\chi\right)  \left\vert _{X_{a,j}}\right.  =\chi$
is a continuous operator from $D_{1,1}\left(  X_{a,j}\right)  $ to
$D_{1,1}\left(  W\right)  $ where $W=\underset{1\leqslant j\leqslant\nu\left(
a\right)  }{\cup}X_{a,j}$. Hence, $u_{j}=u\circ E_{j}$ is a distribution on
$X_{a,j}$. Since $u{}\left\vert _{X_{a,j}^{\ast}}\right.  $ is a usual
harmonic function, $v_{j}=u_{j}\circ z_{j}{}^{-1}$ is a usual harmonic
function in $\mathbb{D}^{\ast}=\mathbb{D}\backslash\left\{  0\right\}  $ which
extends to $\mathbb{D}$ as a distribution. This is possible only if the
holomorphic function $\frac{\partial v_{j}}{\partial z}$ hasn't an essential
singularity at $0$. Indeed, consider $\varepsilon\in\left]  0,1\right[  $,
$\xi\in C_{c}^{\infty}\left(  \left]  \frac{1}{4},\frac{3}{4}\right[
,\mathbb{R}_{+}^{\ast}\right)  $, $p\in\mathbb{Z}\cap\left]  -\infty
,-2\right]  $, $\chi_{\varepsilon}=\chi\left(  \frac{\left\vert z\right\vert
}{\varepsilon}\right)  \left(  \frac{z}{\left\vert z\right\vert }\right)
^{-p}\frac{i}{2}dz\wedge d\overline{z}$ and $\chi_{\varepsilon,j}%
=(z_{j})^{\ast}\chi_{\varepsilon}$. Then $E_{j}\chi_{\varepsilon,j}=0$ on
$\underset{k\neq j}{\cup}X_{a,k}$ and on $\left\{  \left\vert z_{j}\right\vert
\leqslant\frac{\varepsilon}{4}\right\}  $. Hence if\ the Laurent series of
$\frac{\partial v_{j}}{\partial z}$ is $(\Sigma c_{j,n}z^{n})$, we get
\begin{align*}
\left\langle v_{j},\chi_{\varepsilon}\right\rangle  &  =\left\langle
u_{j},\chi_{\varepsilon,j}\right\rangle =\left\langle u,E_{j}\chi
_{\varepsilon,j}\right\rangle =\int_{\left\{  \frac{\varepsilon}{4}%
\leqslant\left\vert z_{j}\right\vert \leqslant\frac{3\varepsilon}{4}\right\}
}u\chi_{\varepsilon}\\
&  =%
{\displaystyle\sum\limits_{n\in\mathbb{Z}}}
\int_{\varepsilon/4}^{3\varepsilon/4}\int_{0}^{2\pi}c_{j,n}\chi\left(
\frac{r}{\varepsilon}\right)  e^{i\pi\left(  n-p\right)  \chi}r^{n+1}drd\chi\\
&  =2\pi c_{j,p}\int_{\varepsilon/4}^{3\varepsilon/4}\chi\left(  \frac
{r}{\varepsilon}\right)  r^{p+1}dr=I_{p}c_{j,p}\varepsilon^{p+2}%
\end{align*}
where $I_{p}=2\pi\int_{1/4}^{3/4}\chi\left(  s\right)  s^{p+1}ds$. On the
other hand, the fact that $v_{j}$ is a distribution on $\mathbb{D}$ implies
that there exists $\left(  C_{j},n_{j}\right)  \in\mathbb{R}_{+}%
\times\mathbb{N}$ such that for any $\theta\in C_{c}^{\infty}\left(  \frac
{3}{4}\overline{\mathbb{D}}\right)  $, $\left\vert \left\langle u_{j}%
,\theta\right\rangle \right\vert \leqslant C_{j}\underset{0\leqslant
m\leqslant n_{j}}{\sup}\left\Vert D^{m}\theta\right\Vert _{\infty}$. As
$\left\Vert \chi_{\varepsilon}\right\Vert _{n_{j}}\leqslant\frac
{Cte}{\varepsilon^{n_{j}}}$, the above last equality implies that $c_{j,p}=0$
if $p<-d_{j}=-n_{j}-2$. Thus, if $\delta_{0}$ denotes the Dirac measure,
$\frac{\partial^{2}v_{j}}{\partial z\partial\overline{z}}=\underset{1\leqslant
n\leqslant d_{j}}{\Sigma}\widetilde{c}_{j,-n}\frac{\partial^{n-1}\delta_{0}%
}{\partial z^{n-1}}$ with $\widetilde{c}_{j,-n}=\pi\frac{\left(  -1\right)
^{n-1}}{\left(  n-1\right)  !}c_{j,-n}$ since $\frac{\partial}{\partial
\overline{z}}\frac{1}{z}=\pi\delta_{0}$.

Assume now that $\chi$ is any smooth function compactly supported in
$\underset{k}{\cap}\left\{  \xi_{k}=1\right\}  $. It follows from the
definition that the $\left(  1,1\right)  $-form $i\partial\overline{\partial
}\chi$ can be written as the sum of the smooth forms $E_{j}(i\partial
\overline{\partial}\chi_{j})$ where $\chi_{j}=\chi\left\vert _{X_{a,j}%
}\right.  $. Hence, setting $\theta_{j}=\chi_{j}\circ z_{j}{}^{-1}$, we get%
\begin{align*}
0  &  =\left\langle i\partial\overline{\partial}u,\chi\right\rangle
=\left\langle u,i\partial\overline{\partial}\chi\right\rangle =%
{\displaystyle\sum\limits_{1\leqslant j\leqslant\nu\left(  a\right)  }}
\left\langle u,E_{j}(i\partial\overline{\partial}\chi_{j})\right\rangle \\
&  =%
{\displaystyle\sum\limits_{1\leqslant j\leqslant\nu\left(  a\right)  }}
\left\langle u_{j},i\partial\overline{\partial}\chi_{j}\right\rangle =%
{\displaystyle\sum\limits_{1\leqslant j\leqslant\nu\left(  a\right)  }}
\left\langle i\partial\overline{\partial}v_{j},\theta_{j}\right\rangle \\
&  =\chi\left(  0\right)
{\displaystyle\sum\limits_{1\leqslant j\leqslant\nu\left(  a\right)  }}
\widetilde{c}_{j,-1}+%
{\displaystyle\sum\limits_{1\leqslant j\leqslant\nu\left(  a\right)  }}
\,%
{\displaystyle\sum\limits_{2\leqslant n\leqslant d_{j}}}
\widetilde{c}_{j,-n}\frac{\partial^{n-1}\theta_{j}}{\partial z^{n-1}}\left(
0\right)
\end{align*}
As $\left(  \frac{\partial^{m}\theta_{j}}{\partial z^{m}}\left(  0\right)
\right)  _{m\geqslant1}$ can be any sequence of complex numbers, the above
equality implies that $c_{j,-n}=0$ when $n\leqslant2$ and $%
{\displaystyle\sum\limits_{1\leqslant j\leqslant\nu\left(  a\right)  }}
c_{j,-1}=0$. As the converse statement of the proposition is clear, the proof
is achieved.
\end{proof}

\subsection{Green functions and Dirichlet problems}

As our proofs use principal Green functions for smooth curves and because
inverse problems require constructive methods, we take the opportunity in this
paper to recall how these functions can be build with constructive tools.
Green conjectured in 1828 the existence of such function for domains in
$\mathbb{R}^{3}$. We recall that a Green function for an open bordered Riemann
surface $Z$ is a symmetric function $g$ defined on $\overline{Z}%
\times\overline{Z}$ without its diagonal such that for any $z\in Z$,
$g_{z}=g\left(  .,z\right)  $ is harmonic on $Z\backslash\left\{  z\right\}
$, smooth on $\overline{Z}\backslash\left\{  z\right\}  $ and has singularity
$\frac{1}{2\pi}\ln\operatorname{dist}\left(  .,z\right)  $ at $z$, the
distance being computed in any hermitian metric on $Z$. It is called principal
if $g_{z}\left\vert _{bZ}\right.  =0$ for any $z\in Z$. The existence of Green
functions for smoothly bordered Riemann surfaces results from classical works
of Fredholm and Hilbert. In \cite{HeG-MiV2012}, an explicit construction has
been derived from Cauchy type formulas even for singular Riemann surfaces. The
theorem below recall how to get a principal Green function from a mundane
one.\smallskip

\noindent\textbf{Theorem}.\textit{\ Let }$Z$\textit{\ be an open bordered
Riemann surface and }$g$\textit{\ a Green function for }$Z$\textit{. We assume
}$\overline{Z}$ \textit{to be contained in some complex curve }$\widetilde{Z}%
$, \textit{e.g. its double. Consider the operator }$T:C^{\infty}\left(
\gamma\right)  \longrightarrow C^{\infty}\left(  \widetilde{Z}\backslash
\gamma\right)  $ \textit{defined by} $Tv:z\mapsto2i\int_{\gamma}%
v\overline{\partial}g_{z}$. \textit{If }$v\in C^{\infty}\left(  \gamma\right)
$\textit{, we denote by }$T^{+}v$\textit{\ the restriction of }$Tv$%
\textit{\ to }$Z^{+}=Z$\textit{\ and by }$T^{-}v$\textit{\ the restriction of
}$Tv$\textit{\ to} $Z^{-}=\widetilde{Z}\backslash Z$. \textit{Then, the
following holds}

\textbf{1. }\textit{For any }$v\in C^{\infty}\left(  \gamma\right)  $\textit{,
}$T^{\pm}v$\textit{\ is harmonic on }$Z^{\pm}$ \textit{and extends
continuously to }$\gamma$.

\textbf{2. }\textit{(Sohotsky-1873 when} $Z\subset\mathbb{C}$) \textit{For any
}$v\in C^{\infty}\left(  \gamma\right)  $, $v=\left(  T^{+}v\right)
\left\vert _{\gamma}\right.  -\left(  T^{-}v\right)  \left\vert _{\gamma
}\right.  $.

\textbf{3. }\textit{(Fredholm-1899 when} $Z\subset\mathbb{C}$) \textit{For any
}$u\in C^{\infty}\left(  \gamma\right)  $,\textit{\ the unique harmonic
extension }$Eu$\textit{\ of }$u$\textit{\ to }$Z$\textit{\ is the solution
}$v$\textit{\ of the integral equation} $u=v+\left(  T^{-}v\right)  \left\vert
_{\gamma}\right.  $.

\textbf{4. }\textit{The principal Green function for }$Z$\textit{\ is the
function }$G$\textit{\ defined by }$G\left(  z,\zeta\right)  =g\left(
z,\zeta\right)  -\left(  Eg_{z}\left\vert _{\gamma}\right.  \right)  \left(
\zeta\right)  $\textit{\ for all }$\left(  z,\zeta\right)  \in Z^{2}%
$\textit{\ with }$z\neq\zeta$\textit{.}\smallskip

In the sequel, an \textit{admissible family} for an open bordered nodal curve
$X$ is a family $\left(  c_{a,j}\right)  _{a\in\operatorname*{Sing}%
X,~1\leqslant j\leqslant\nu\left(  a\right)  }$ of complex numbers such that
for each node $a$ of $X$, $%
{\displaystyle\sum\limits_{1\leqslant j\leqslant\nu\left(  a\right)  }}
c_{a,j}=0$. The following proposition generalizes the initial statement of
Riemann and Klein given in the introduction. It is a\ consequence of the above
classical result and proposition~\ref{SingLog}.

\begin{proposition}
[Solution of the nodal Dirichlet problem]\label{Dir}Let $X$ be an open
bordered or compact nodal curve, $c=\left(  c_{a,j}\right)  _{a\in
\operatorname*{Sing}X,~1\leqslant j\leqslant\nu\left(  a\right)  }$ an
admissible family and for each sufficiently small branch $X_{a,j}$ at a node
$a$ in $X$, let us fix some holomorphic coordinate $z_{j}$ for $X_{a,j}$
centered at $a$. If $X$ is non compact, we also fix $u\in C^{0}\left(
bX\right)  $. Then there exists a unique (up to an additive constant if $X$ is
compact) harmonic distribution $\widetilde{u}^{c}$ on $X$ such that
$\widetilde{u}^{c}\left\vert _{bX}\right.  =u$ (if $X$ is non compact) and
$\widetilde{u}^{c}\left\vert _{X_{a,j}^{\ast}}\right.  -2c_{a,j}\ln\left\vert
z_{j}\right\vert $ extends as a usual harmonic function near $a$ in $X_{a,j}$.
Equivalently, $\widetilde{u}^{c}$ is the harmonic distribution $U$ extending
$u$ to $X$ such that $\partial U$ is a meromorphic (1,0)-form whose
singularities are simple poles at nodes of $X$ with residue $c_{a,j}$ along
$X_{a,j}$ when $a\in\operatorname*{Sing}X$.
\end{proposition}

\begin{proof}
[Proof]Assume $X$ is non compact. Let then $\widehat{g}$ be a Green function
for a normalization $\widehat{X}\overset{\pi}{\longrightarrow}X$ of $X$ such
that $\widehat{g}_{\zeta}\overset{def}{=}\widehat{g}\left(  \zeta,.\right)
=0$ on $b\widehat{X}$ for any $\zeta\in\widehat{X}$. As $\overline{X}$ is a
smooth manifold with boundary near $bX$, $b\widehat{X}\overset{\pi
}{\longrightarrow}bX$ is a diffeomorphism and $v=\pi^{\ast}u$ is a well
defined continuous function on $b\widehat{X}$. Let $V$ the distribution
defined on $\widehat{X}$ by%
\[
V=\widetilde{v}+%
{\displaystyle\sum\limits_{a\in\operatorname*{Sing}X}}
\,%
{\displaystyle\sum\limits_{1\leqslant j\leqslant\nu\left(  a\right)  }}
c_{a,j}\widehat{g}_{a_{j}}%
\]
where $\widetilde{v}$ is the harmonic extension of $v$ to $\widehat{X}$ and
where for each $a\in\operatorname*{Sing}X$, $\left\{  a_{1},...,a_{\nu\left(
a\right)  }\right\}  =\pi^{-1}\left(  a\right)  $ and $X_{a,j}=\pi\left(
W_{j}\right)  $, $W_{j}$ being some neighborhood of $a_{j}$ in $\widehat{X}$.
Then $\widetilde{u}^{c}=\pi_{\ast}V$ is a distribution on $X$ which is a usual
harmonic function on $\operatorname*{Reg}\overline{X}$ that extends $u$. The
same kind of computing as in proposition~\ref{SingLog} shows that
$\widetilde{u}^{c}$ is a harmonic distribution on $X$. Since $\overline{X}$
has smooth boundary, $\widetilde{u}^{c}$ has the same regularity as $u$ in a
neighborhood of $bX$ in $\overline{X}$.

To prove uniqueness, we have to show that if $U$ is a harmonic distribution
$X$ which vanish on $bX$ and has usual harmonic extension to any branch of
$X$, then $U=0$. Let us consider such an $U$. Then $V=\pi^{\ast}U$ is a well
defined function~; if $a_{j}\in\pi^{-1}\left(  a\right)  $ is in the closure
of $\pi^{-1}\left(  X_{a,j}^{\ast}\right)  $, then $V\left(  a_{j}\right)  $
is the value at $a$ of the harmonic extension of $U\left\vert _{X_{a,j}^{\ast
}}\right.  $. $V$ is of course harmonic on $\widehat{X}$, continuous up to the
boundary and vanish on it. So $V=0$. Hence $U=0$.

When $X$ is compact the classical construction techniques of bipolar Green
functions can be adapted to get on a normalization of $X$ multipolar Green
functions which, thanks to the properties of admissible families, can be seen
as harmonic distributions on $X$.
\end{proof}

The proposition~\ref{Dir} shows in particular that any continuous function on
the boundary of $X$ has many (weakly) harmonic distribution extensions to $X$
when no datum is specified at nodes. This non uniqueness phenomenon also
occurs for principal Green functions on nodal surfaces.

\section{Inverse problems for compact and nodal curves}

The inverse Dirichlet to Neumann problem (IDN problem for short) for a given
(smooth) Riemann surface $X$ with smooth boundary $\gamma$ is to reconstruct
it from the data of $\gamma$, $T_{\gamma}X$ and its Dirichlet to Neumann
operator which is the operator associating to a smooth function on $\gamma$
the restriction on $\gamma$ of the normal derivative of its harmonic extension
to $X$. This subject has been started by Belishev and
Kurylev~\cite{BeM-KuY1992} in a non stationary setting. For the stationary
case, uniqueness results based on the full knowledge of the DN-operator are
obtained in \cite{LaM-UhG2001} and~\cite{BeM2003}. The constructive
reconstruction method given in \cite{HeG-MiV2007} is here extended to Riemann
surfaces, compact or nodal.

\subsection{DN-data~; hypothesis A and B}

Let $X$ be an open bordered nodal curve. Since $\overline{X}$ has smooth
boundary, we can select two vector fields along $bX$, $\tau$ and $\nu$, such
that $\tau$ is a smooth generating section of the tangent bundle $T(bX) $ of
$bX$ and for each $x$ in $bX$, $\left(  \nu_{x},\tau_{x}\right)  $ is a
positively oriented orthonormal basis of $T_{x}\overline{X}$. Then, the
\textit{Dirichlet-Neumann operator} for $X$ and some admissible family
$c=\left(  c_{a,j}\right)  _{a\in\operatorname*{Sing}X,~1\leqslant
j\leqslant\nu\left(  a\right)  }$ is the the operator $N_{X,c}$ defined for
any $u\in C^{1}\left(  bX\right)  $ by
\[
N_{X,c}u=\left.  \frac{\partial\widetilde{u}^{c}}{\partial\nu}\right\vert
_{bX}%
\]
where $\widetilde{u}^{c}$ is the extension of $u$ to $X$ as a harmonic
distribution such that $\partial\widetilde{u}^{c}$ has residue $c_{a,j}$ at
$a$ when $a\in\operatorname*{Sing}X$ and $1\leqslant j\leqslant\nu\left(
a\right)  $.

Since an admissible family don't reflect the complex structure of $X$ but only
tracks nodes' existence, a natural inverse Dirichlet to Neumann problem is to
look for a process rebuilding $X$ from the data of $bX$ and the action of some
$N_{X,c}$ on some $u\in C^{1}\left(  bX\right)  $, where $c$ belongs to an
unknown set of admissible families.

Whether or not admissible families can be recovered from boundary data is a
very natural question if one considers the physical origin of the problem as
explained in the introduction~; they corresponds to the charges set up on the
nodes.\medskip

With these inverse reconstruction problems arise the question of uniqueness of
an open bordered nodal curve having a given boundary data. So let $\gamma$ be
a smooth compact oriented real curve without component reduced to a point. Let
$\tau$ be a smooth generating section of $T\gamma$ and $\nu$ an another vector
field along $\gamma$ such that the bundle $\mathcal{T}$ generated by $\left(
\nu_{x},\tau_{x}\right)  _{x\in\gamma}$, has rank $2$~; $\gamma$ is assumed to
be oriented by $\tau$ and $\mathcal{T}$ by $\left(  \nu,\tau\right)  $.
Consider an operator $N$ from $C^{1}\left(  \gamma\right)  $ to the space of
currents on $\gamma$ of degree $0$ and order $1$ (i.e. functionals on $C^{1}$
1-forms on $\gamma$). As in \cite{HeG-MiV2007}, we use a setting which
emphasizes the involved complex analysis. With $N$ come two other operators
$L$ and $\theta$ defined for $u\in C^{1}\left(  \gamma\right)  $ by%
\begin{equation}
Lu=\frac{1}{2}\left(  Nu-i\,Tu\right)  \hspace{0.5cm}\&\hspace{0.5cm}\theta
u=\left(  Lu\right)  \left(  \nu^{\ast}+i\tau^{\ast}\right)
\label{F/ N vers CR et theta}%
\end{equation}
where $T$ is the tangential derivation by $\tau$ and $(\nu_{x}^{\ast},\tau
_{x}^{\ast})$ is the dual basis of $\left(  \nu_{x},\tau_{x}\right)  $ for
every $x\in\gamma$.

Note that if actually $\gamma$ is the smooth boundary of bordered nodal curve
$X$ such that $\left(  \nu_{x},\tau_{x}\right)  $ is a positively oriented
orthonormal basis of $T_{x}\overline{X}$, the equality $Nu=N_{X,c}u$ is
equivalent to the identity $\left(  \partial\widetilde{u}^{c}\right)
\left\vert _{\gamma s}\right.  =\theta u$.

In the smooth case, we know from \cite{HeG-MiV2007} that the knowledge of
$\gamma$ and the action of $N$ on only three generic (in the sense detailed
hereafter) continuous functions is sufficient to reconstruct such a Riemann
surface when it exists. Since an admissible family do not encode the complex
structure of $X$, it is natural to let the given boundary data corresponds to
different admissible families and hence to different operators $N$. So, we are
lead to consider the following~:\smallskip

\noindent\textbf{A. }We consider three operators $N_{0}$, $N_{1}$, and $N_{2}$
from $C^{1}\left(  \gamma\right)  $ to the space of currents on $\gamma$ of
degree $0$ and order $1$, their corresponding operators $\theta_{\ell}=\left(
\nu^{\ast}+i\tau^{\ast}\right)  L_{\ell}$ where $L_{\ell}=\frac{1}{2}\left(
N_{\ell}-i\,T\right)  $ and $u_{0},u_{1},u_{2}\in C^{\infty}\left(
\gamma\right)  $ three real valued functions only ruled by the hypothesis that%
\begin{equation}
f=\left(  f_{1},f_{2}\right)  =\left(  \,\left(  L_{\ell}u_{\ell}\right)
/\left(  L_{0}u_{0}\right)  \,\right)  _{\ell=1,2}=\left(  \,\left(
\theta_{\ell}u_{\ell}\right)  /(\theta_{0}u_{0})\,\right)  _{\ell=1,2}
\label{F/ f}%
\end{equation}
is an embedding of $\mathbb{\gamma}$ in $\mathbb{C}^{2}$ considered as the
complement of $\left\{  w_{0}=0\right\}  $ in the complex projective plane
$\mathbb{CP}_{2}$ with homogeneous coordinates $\left(  w_{0}:w_{1}%
:w_{2}\right)  $. This is somehow generic since the
proposition~\ref{P/ generique B} below shows in particular that if it happens
that $\gamma$ is the smooth boundary of a complex curve and $N$ are Dirichlet
to Neumann operators, then the set of $\left(  u_{\ell}\right)  _{0\leqslant
\ell\leqslant2}\in C^{\infty}\left(  \gamma\right)  ^{3}$ such that the above
map $f$ is an embedding is a dense open set of $C^{\infty}\left(
\gamma\right)  ^{3}$.\smallskip

Before defining what are restricted DN-data, we have to precise how a 3-uple
$\omega=\left(  \omega_{0},\omega_{1},\omega_{2}\right)  $ of smooth $\left(
1,0\right)  $-forms which never vanish simultaneously induces a map from $X$
to $\mathbb{CP}_{2}$~; such 3-uples exist since a normalization of $X$ do have
ones. We define a map from $X$ to $\mathbb{CP}_{2}$, denoted $\left[
\omega\right]  $ or $\left(  \omega_{0}:\omega_{1}:\omega_{2}\right)  $, by
defining it with the formulas $\left[  \omega\right]  =\left(  1:\frac
{\omega_{1}}{\omega_{0}}:\frac{\omega_{2}}{\omega_{0}}\right)  $ on $\left\{
\omega_{0}\neq0\right\}  $, $\left[  \omega\right]  =\left(  \frac{\omega_{0}%
}{\omega_{1}}:1:\frac{\omega_{2}}{\omega_{1}}\right)  $ on $\left\{
\omega_{1}\neq0\right\}  $ and $\left[  \omega\right]  =\left(  \frac
{\omega_{0}}{\omega_{2}}:\frac{\omega_{1}}{\omega_{2}}:1\right)  $ on
$\left\{  \omega_{2}\neq0\right\}  $~; here each quotient written is a well
defined (along any branch) meromorphic function since $\dim X=1$.
In~\cite{HeG-MiV2007}, we have identified with a slight abuse of language
$\left[  \omega\right]  $ with $\left[  \omega\right]  \left\vert _{\left\{
\omega_{0}\neq0\right\}  }\right.  $ which have the affine coordinates
$\left(  \frac{\omega_{1}}{\omega_{0}},\frac{\omega_{2}}{\omega_{0}}\right)  $
in the affine hyperplane $\left\{  w_{0}\neq0\right\}  $.\smallskip

Assume now that $u=\left(  u_{\ell}\right)  _{0\leqslant\ell\leqslant2}\in
C^{\infty}\left(  \gamma\right)  ^{3}$ and set $\theta u=\left(  \theta_{\ell
}u_{\ell}\right)  _{0\leqslant\ell\leqslant2}$. We call $\left(
\gamma,u,\theta u\right)  $ a \textit{restricted DN-datum} for an open
bordered nodal curve $X$ if $\left(  \gamma,u,\theta u\right)  $ satisfies (A)
and the following~:\smallskip

\noindent\textbf{B1. }$X$ has smooth boundary $\gamma$.

\noindent\textbf{B2. }For each $\ell\in\left\{  0,1,2\right\}  $,
$\theta_{\ell}u_{\ell}=\left(  \partial\widetilde{u_{\ell}}^{c_{\ell}}\right)
{}\left\vert _{\gamma}\right.  $ for some admissible family $c_{\ell}$.

\noindent\textbf{B3. }The $\partial\widetilde{u_{\ell}}^{c_{\ell}}$ have no
common zero and the map
\[
F=\left(  \partial\widetilde{u_{0}}^{c_{0}}:\partial\widetilde{u_{1}}^{c_{1}%
}:\partial\widetilde{u_{2}}^{c_{2}}\right)
\]
extends to $X$ the map $\left(  1:f_{1}:f_{2}\right)  $ defined by
(\ref{F/ f}) in the sense that for every $x_{0}\in\gamma$, $\underset
{x\rightarrow x_{0},~x\in X}{\lim}\left(  \frac{\partial\widetilde{u_{1}%
}^{c_{1}}}{\partial\widetilde{u_{0}}^{c_{0}}}\left(  x\right)  ,\frac
{\partial\widetilde{u_{2}}^{c_{2}}}{\partial\widetilde{u_{0}}^{c_{0}}}\left(
x\right)  \right)  $ exists and equals $\left(  f_{1}\left(  x_{0}\right)
,f_{2}\left(  x_{0}\right)  \right)  $~; this last property holds
automatically if $\gamma$ and $f$ are real analytic.

\noindent\textbf{B4. }There is a finite subset\ $A$ in $X$ such that $F$ is an
embedding from $Z\backslash A$ into $\mathbb{CP}_{2}$.\smallskip

If one wish to emphasize the admissible families $c_{\ell}$, we say that
$\left(  \gamma,u,\theta u\right)  $ and the $c_{\ell}$ are associated. When
each node $a_{n}$, $1\leqslant n\leqslant N$, of $X$ is obtained by
identification in a Riemann surface $Z$ of the points in a family $\left(
a_{n,j}\right)  _{1\leqslant j\leqslant\nu_{n}}$ and when the family of
charges or residues corresponding to $a_{n}$ and $u_{\ell}$ is $\left(
c_{\ell,n,j}\right)  _{1\leqslant j\leqslant\nu_{n}}$ ($0\leqslant
\ell\leqslant2$) we also say that $\left(  \gamma,u,\theta u\right)  $ is a
restricted DN-datum for $Z$ and the $a_{n,j}$ charged by $c_{\ell,n,j}$, or
$c_{n,j}$ if no $c_{\ell,n,j}$ depends on $\ell$.

The condition (B4) may seem restrictive but is open and dense in the following sense~:

\begin{proposition}
\label{P/ generique B}Assume $\gamma$ is the boundary \ of an open bordered
nodal Riemann surface $X$, $c=\left(  c_{\ell}\right)  _{0\leqslant
\ell\leqslant2}$ is a 3-uple of admissible families and $u=\left(  u_{\ell
}\right)  _{0\leqslant\ell\leqslant2}$ satisfies hypothesis A. Then the set
$E_{X,c}$ of $u$ in $C^{\infty}\left(  \gamma\right)  ^{3}$ such that $\left(
\gamma,u,\theta u\right)  $ is a restricted DN-datum for $X$ and $c$ is a
dense open set of $C^{\infty}\left(  \gamma\right)  ^{3}$.
\end{proposition}

As stated in the introduction, this result as well as
proposition~\ref{P/ compact generique} will be proved in a separate paper
because they involve methods and results of complex analysis which deserve
attention of their own.\smallskip

When $X$ is smooth, there is no node and the only 3-uple of admissible
families is the empty one. Dropping in the above definition any reference to
admissible families gives a restricted DN-datum notion in the smooth case.
Meanwhile, because of (B4), restricted DN-data thus defined here are more
specific than those considered in \cite{HeG-MiV2007}. Actually, in the
exceptional case $\left(  \gamma,u,\theta u\right)  $ satisfies only (B1) to
(B3), it is possible that the direct image by $F$ of the integration current
on $X$ is not always, contrary to~\cite[lemma~7]{HeG-MiV2007} proof's claim,
an integration current over a subvariety of $\mathbb{CP}_{2}\backslash
f\left(  \gamma\right)  $. The reason is that in the general case, $F_{\ast
}\left[  X\right]  $ could even not be a locally flat current as defined in
\cite{FeH1969Li}. However, all statements of \cite{HeG-MiV2007} are true with
the above reinforced definition of restricted DN-datum.\smallskip

It follows from the definitions that if $X$ is an open bordered nodal curve
obtained after identification of some points in an open bordered Riemann
surface $Z$ and $\pi:Z\longrightarrow X$ is the natural projection, the direct
image by $\pi$ of any harmonic function on $Z$ continuous up to $\overline{Z}$
is a harmonic distribution on $X$ solving a Dirichlet problem with a zero
admissible family. Hence, the data of its differential along $bX$ fail to
encode any information about the nodal curve but its normalization. This
motivates the following definition. We say that a finite family $\left(
w_{s}\right)  _{s\in\Sigma}$ of complex numbers is \textit{generic for a
partition} $\left\{  \Sigma_{1},..,\Sigma_{N}\right\}  $ of $\Sigma$ if the
following holds~:

\begin{itemize}
\item $%
{\displaystyle\sum\limits_{s\in\Sigma_{j}}}
w_{s}=0$ for any $j$.

\item For any family of sets $\left(  T_{j}\right)  _{1\leqslant j\leqslant
N}$ such that $T_{j}\subsetneq\Sigma_{j}$ for all $j$ and $\underset
{1\leqslant j\leqslant N}{\cup}T_{j}\neq\varnothing$, $%
{\displaystyle\sum\limits_{1\leqslant j\leqslant N}}
$ $%
{\displaystyle\sum\limits_{t\in T_{j}}}
w_{t}\neq0$.
\end{itemize}

\noindent An admissible family $c=\left(  c_{a,j}\right)  _{\left(
a,j\right)  \in\Sigma}$ ($\Sigma=\underset{a\in\operatorname*{Sing}X}{\cup
}\Sigma_{a}$, $\Sigma_{a}=\left\{  a\right\}  \times\left\{  1,...,\nu\left(
a\right)  \right\}  $) of an open bordered nodal curve $X$ is said to be
\textit{generic} \textit{for} $X$ if it is generic for $\left\{  \Sigma
_{a};~a\in\operatorname*{Sing}X\right\}  $, that is, when the only way to
achieve $%
{\displaystyle\sum\limits_{a\in\operatorname*{Sing}X}}
{\displaystyle\sum\limits_{j\in J_{a}}}
c_{a,j}=0$ with $J_{a}\subset\left\{  1,...,\nu\left(  a\right)  \right\}  $
for all $a\in\operatorname*{Sing}X$ is either to have $J_{a}=\left\{
1,...,\nu\left(  a\right)  \right\}  $ or $J_{a}=\varnothing$ for all $a$.

\subsection{Proofs of results for the compact case}

\subsubsection{$\frac{{}}{{}}$Proof of theorem~\ref{T/ compactR3opt}}

We assume with no loss of generality that $S$ is smoothly bordered and
$\gamma=bS$ is then equipped with the orientation induced by $Z\backslash S$.
Set $u=\left(  U_{Z,\ell}^{a,c}\left\vert _{\gamma}\right.  \right)
_{0\leqslant\ell\leqslant2}$ and $\theta u=\left(  \partial_{Z}U_{Z,\ell
}^{a,c}\left\vert _{\gamma}\right.  \right)  _{0\leqslant\ell\leqslant2}$
(resp. $u^{\prime}=\left(  U_{Z^{\prime},\ell}^{a,c}\left\vert _{\gamma
}\right.  \right)  _{0\leqslant\ell\leqslant2}$ and $\theta^{\prime}u^{\prime
}=\left(  \partial_{Z^{\prime}}U_{Z^{\prime},\ell}^{a,c}\left\vert _{\gamma
}\right.  \right)  _{0\leqslant\ell\leqslant2}$) where $\overline{\partial
}_{Z}$ and $\overline{\partial}_{Z^{\prime}}$ are the Cauchy-Riemann operators
of $Z$ and $Z^{\prime}$. By hypothesis, the complex structures on $S$ induced
by $Z$ and $Z^{\prime}$ are the same, namely the complex structure induced by
the standard metric of $\mathbb{R}^{3}$ on $S$. Hence, $\left(  \gamma
,u,\theta u\right)  =\left(  \gamma,u^{\prime},\theta^{\prime}u^{\prime
}\right)  $ follows from $\left(  U_{Z,\ell}^{a,c}\left\vert _{S}\right.
\right)  _{0\leqslant\ell\leqslant2}=\left(  U_{Z^{\prime},\ell}%
^{a,c}\left\vert _{S}\right.  \right)  _{0\leqslant\ell\leqslant2}$. As
$\left(  a,c\right)  $ is assumed to be in $E_{Z}\cap E_{Z^{\prime}}$, it
appears that $\left(  \gamma,u,\theta u\right)  $ is a restricted datum for
both $Z\backslash S$ and $Z^{\prime}\backslash S$. Theorem~1 of
\cite{HeG-MiV2007} with the above definition of DN-datum or
theorem~\ref{T/ unicite} below applies and gives us the existence of an
isomorphism $\varphi_{1}:Z\backslash S\longrightarrow Z^{\prime}\backslash S$
which is the identity on $bS$. Thanks to the Morera theorem, the gluing of
$\varphi_{1}$ with $Id_{S}$ gives the desired isomorphism $\varphi
:Z\longrightarrow Z^{\prime}$.

The reconstruction part follows directly from \cite[ th.~2 ]{HeG-MiV2007}
applied to $Z\backslash S$. The reconstruction formulas are the same that
those of theorem~\ref{T/ looseIDN} which is an adaptation to the nodal case of
\cite[ th.~2 ]{HeG-MiV2007}.\smallskip

\noindent\textbf{Remark. }If $Z$ and $Z^{\prime}$ are actually compact nodal
curves and if the potentials are associated to generic admissible families,
the same proof readily applies thanks to theorems~\ref{T/ unicite}
and~\ref{T/ looseIDN}.

\subsubsection{Proof of theorem~\ref{T/ compactR3}}

Under the hypothesis of theorem~\ref{T/ compactR3}, the proof\ of
theorem~\ref{T/ compactR3opt} readily apply.

\subsection{Uniqueness and reconstruction results for nodal curves}

The theorem~\ref{T/ unicite} below shows what kind of uniqueness can be
expected in the inverse Dirichlet to Neumann problem.

\begin{theorem}
[Uniqueness for IDN problems]\label{T/ unicite}Assume that $X$ and $X^{\prime
}$ are bordered nodal curves with a restricted DN-datum $\left(
\gamma,u,\theta u\right)  $ associated to admissible families $c_{0}%
,c_{1},c_{2}$ for $X$ and to admissible families $c_{0}^{\prime},c_{1}%
^{\prime},c_{2}^{\prime}$ for $X^{\prime}$. Then, the following holds.

1. $X$ and $X^{\prime}$ are obtained by the identification of some finite sets
of points in a same open bordered Riemann surface.

\textbf{2. }If at least one of the admissible families associated to $\left(
\gamma,u,\theta u\right)  $ has no zero coefficient, $X\cup\gamma$ and
$X^{\prime}\cup\gamma$ are roughly isomorphic through a map which is the
identity on $\gamma$.

\textbf{3. }If at least one of the $c_{\ell}$ and one of the $c_{\ell}%
^{\prime}$ is generic for $X$ and $X^{\prime}$ respectively, then there is an
isomorphism of bordered nodal curves between $X\cup\gamma$ and $X^{\prime}%
\cup\gamma$ whose restriction on $\gamma$ is the identity.
\end{theorem}

\noindent\textbf{Remarks. 1. }If $E\subset\gamma$ and $h^{1}\left(  E\cap
c\right)  >0$ for each connected component $c$ of $\gamma$, meromorphic
functions are uniquely determined by their values on $E$ and it follows that
the theorem~\ref{T/ unicite} conclusions hold when $N_{X^{\prime},c^{\prime}%
}u_{\ell}=N_{X,c}u_{\ell}$ is ensured only on $E$ and the meromorphic
functions $\left(  \partial\widetilde{u_{\ell}}^{c_{\ell}}\right)  /\left(
\partial\widetilde{u}_{0}^{c_{0}}\right)  $ and $(\partial\widetilde{u_{\ell}%
}^{c_{\ell}^{\prime}})/(\partial\widetilde{u}_{0}^{c_{0}^{\prime}})$ are
continuous near $\gamma$.

\textbf{2. }Assertion\textbf{\ }(2) shows that when two bordered nodal curves
$X$ and $X^{\prime}$ share a same \textit{restricted DN-datum} $\left(
\gamma,u,\theta u\right)  $, there is only a finite indeterminacy between $X$
and $X^{\prime}$.

\textbf{3. }If one of the $c_{\ell}$ is generic for $X$, then there is a
holomorphic surjective map from $X$ onto $X^{\prime}$, that is a continuous
map holomorphic along any branch of $X$ which sends a node of $X$ to a node of
$X^{\prime}$. Hence, $X^{\prime}$ may be considered as a quotient of $X$ and,
equivalently, $X$ as a partial normalization of $X^{\prime}$.

\begin{proof}
[Proof]Let $Z\overset{\pi}{\longrightarrow}X$ be a normalization of $X$ and
$g=f\circ\pi$ where $f$ is defined by (\ref{F/ f}). Then $\widehat{\gamma}%
=\pi^{-1}\left(  \gamma\right)  $ is a smooth oriented real curve and bounds
smoothly $\overline{Z}$. Thanks to hypothesis A, $\delta=f\left(
\gamma\right)  =g\left(  \widehat{\gamma}\right)  $ is a smooth compact
oriented real curve of $\mathbb{CP}_{2}$ without component reduced to a point.
It follows from (B2) that $G=F\circ\pi$ is a meromorphic extension of $g$ to
$Z$. To prove (1), we now follow the proof of \cite[th.~1]{HeG-MiV2007} for
which the key point is that the embedding $g$ into $\mathbb{CP}_{2}$ of the
real curve $\widehat{\gamma}$ extends meromorphically as $G$ to the complex
curve $Z$. The gap arising in conditions of \cite[lemma 7]{HeG-MiV2007} is
avoided thanks to the reinforced but still generic definition of restricted DN-datum.

Since $\left(  \gamma,u,\theta u\right)  $ satisfies (B4) and $X$ is nodal,
there is in $Z$ a finite set $A$ such that $G$ is an isomorphism from
$Z\backslash A$ to $G\left(  X\right)  \backslash G\left(  A\right)  $. Thus
$Y=G\left(  X\right)  \backslash\delta$ has to be a complex curve of
$\mathbb{C}\mathbb{P}_{2}\backslash\delta$ satisfying $d\left[  Y\right]
=\left[  \delta\right]  $. It contains no compact complex curve because $X$
has none and it has finite mass because of a theorem of Wirtinger (see
\cite{HaR1977}). Note that $A$ may meet $G^{-1}\left(  \delta\right)  $ but
that each point of $\overline{Z}$ has in $\overline{Z}$ a neighborhood $V$
such that $F:V\rightarrow F\left(  V\right)  $ is diffeomorphism between
manifolds with smooth boundary. Hence, the conclusions of \cite[lemma
7]{HeG-MiV2007} are valid.

Set $Z_{\circ}=Z\backslash G^{-1}\left(  \delta\right)  $,
$B=\operatorname*{Sing}\overline{Y}$ and $\widehat{A}=G{}^{-1}\left(
B\right)  $~; note that $S=\pi^{-1}\left(  \operatorname*{Sing}X\right)
\subset\widehat{A}$. The proof of lemma~9 of \cite{HeG-MiV2007} gives that
$G:Z_{\circ}\backslash\widehat{A}\rightarrow Y\backslash B=\operatorname*{Reg}%
Y\ $is an isomorphism of complex manifolds, $G:\overline{Z}\rightarrow
\overline{\widehat{Y}}$ is proper and that $G:\overline{Z}\backslash
A\rightarrow\overline{Y}\backslash B=\operatorname*{Reg}\overline{Y}\ $is an
isomorphism of manifolds with smooth boundaries. The same construction apply
for $X^{\prime}$. Denoting by a prime every preceding notation above to get
objects related to $X^{\prime}$, we can apply the end of the proof of
\cite[th.~1]{HeG-MiV2007} and get that $\overline{Z}$ and $\overline
{Z^{\prime}}$ are isomorphic through a map $\Phi:X\longrightarrow X^{\prime}$
which is the identity on $\gamma$. Thus, the first point of the theorem is proved.

The $\left(  1,0\right)  $-forms $U_{1}=\pi^{\ast}\partial\widetilde{u_{1}%
}^{c_{1}}$ and $U_{1}^{\prime}=\Phi^{\ast}\pi^{\prime\ast}\partial
\widetilde{u_{1}}^{c_{1}^{\prime}}$ are meromorphic on $Z$ and they are equal
on $\gamma=bZ$. Hence, they are equal on $Z$. In particular, they have same
poles and same residues. If $c_{1}$ or $c_{1}^{\prime}$ has no zero
coefficient, this yield $\Phi^{-1}\left(  S^{\prime}\right)  =S$. Thus, the
second point of the theorem is proved.

Set $\Sigma=\underset{a\in\operatorname*{Sing}X}{\cup}\left\{  a\right\}
\times\left\{  1,...,\nu\left(  a\right)  \right\}  $ and $\Sigma^{\prime
}=\underset{a^{\prime}\in\operatorname*{Sing}X}{\cup}\left\{  a^{\prime
}\right\}  \times\left\{  1,...,\nu\left(  a^{\prime}\right)  \right\}  $. For
each $a$ in $\operatorname*{Sing}X$ (resp. each $a^{\prime}$ in
$\operatorname*{Sing}X^{\prime}$), we fix a numeration $\left(  X_{a,j}%
\right)  _{1\leqslant j\leqslant\nu\left(  a\right)  }$ (resp. $\left(
X_{a^{\prime},j^{\prime}}^{\prime}\right)  _{1\leqslant j^{\prime}\leqslant
\nu\left(  a^{\prime}\right)  }$) of the branches of $X$ (resp. $X^{\prime}$)
at $a$ (resp. $a^{\prime}$). Let $\left(  a,j\right)  $ be in $\Sigma$ and
$W_{a,j}=\pi^{-1}\left(  X_{a,j}\right)  $. Then the restriction $\pi_{j}%
=\pi\left\vert _{W_{a,j}}^{X_{a,j}}\right.  $ is bijective so we can define
$s_{j}=\pi_{j}^{-1}\left(  a\right)  \in S$, $s_{j}^{\prime}=\Phi\left(
s_{j}\right)  $, $b\left(  a,j\right)  =\pi^{\prime}\left(  s^{\prime}\right)
=\pi^{\prime}\left(  \Phi_{j}\left(  \pi_{j}^{-1}\left(  a\right)  \right)
\right)  $ and $\sigma\left(  a,j\right)  =\left(  b\left(  a,j\right)
,k\left(  a,j\right)  \right)  $ where $k\left(  a,j\right)  $ is the integer
such that $\pi^{\prime}\left(  \Phi\left(  W_{a,j}\right)  \right)  $ and
$X_{b\left(  a,j\right)  ,k\left(  a,j\right)  }^{\prime}$ have the same germ
at $b\left(  a,j\right)  $. The map $\sigma:$ $\Sigma\longrightarrow
\Sigma^{\prime}$ is bijective by construction~; we set $\sigma^{-1}=\left(
b^{\prime},k^{\prime}\right)  $. The map $\varphi=\pi^{\prime}\circ\Phi
\circ\pi_{r}^{-1}$ ($\pi_{r}\overset{def}{=}\pi\left\vert _{\pi^{-1}\left(
\operatorname*{Reg}X\right)  }^{\operatorname*{Reg}X}\right.  $) extends as a
multivaluate map, still denoted by $\varphi$, continuous along each branch of
$X$ if $\varphi\left\vert _{X_{a,j}^{\ast}}\right.  $ is prolonged at $a$ by
the value $b\left(  a,j\right)  $. Likewise, we denote by $\psi$ the
multivaluate extension of $\pi\circ\Phi^{-1}\circ\pi_{r}^{\prime-1}$ ($\pi
_{r}^{\prime}\overset{def}{=}\pi^{\prime}\left\vert _{\pi^{\prime-1}\left(
\operatorname*{Reg}X^{\prime}\right)  }^{\operatorname*{Reg}X^{\prime}%
}\right.  $) such that for any $\left(  a^{\prime},j^{\prime}\right)
\in\Sigma^{\prime}$, $\psi\left\vert _{X_{a^{\prime},j^{\prime}}}\right.  $ is
continuous and, hence, take the value $b^{\prime}\left(  a^{\prime},j^{\prime
}\right)  $ at $\left(  a^{\prime},j^{\prime}\right)  $. To get that $X$ and
$X^{\prime}$ are actually isomorphic, we check that if $a$ is a given node of
$X$, $b\left(  a,1\right)  =\cdots=b\left(  a,\nu\left(  a\right)  \right)
\overset{def}{=}a^{\prime}$ and $\nu\left(  a^{\prime}\right)  =\nu\left(
a\right)  $.

Assume now that $c_{1}$ is generic for $X$. Let $a^{\prime}$ be a node of
$X^{\prime}$. The family $\left(  X_{a^{\prime},j^{\prime}}^{\prime}\right)
_{1\leqslant j^{\prime}\leqslant\nu^{\prime}\left(  a^{\prime}\right)  }$ of
its branches can be part in disjoint families $\left(  X_{a^{\prime}%
,j^{\prime}}^{\prime}\right)  _{j^{\prime}\in J_{a^{\prime}}^{m}}$,
$1\leqslant m\leqslant\mu$, such that , for each $m$ and each $j^{\prime}\in
J_{a^{\prime}}^{m}$, $\left\{  a_{m}\right\}  \cup\varphi^{-1}(X_{a^{\prime
},j^{\prime}}^{\prime\ast})$ is some branch $X_{a_{m},\ell_{m}\left(
j^{\prime}\right)  }$ of $X$ at some node $a_{m}$ of $X$. Since $\widetilde
{u_{1}}^{c_{1}^{\prime}}$ is a harmonic distribution solving a Dirichlet
problem associated to the admissible family $\left(  c_{1}^{\prime}\right)  $,
we know that
\[
0=%
{\displaystyle\sum\limits_{1\leqslant j\leqslant\nu^{\prime}\left(  a^{\prime
}\right)  }}
c_{1,a^{\prime},j}^{\prime}=%
{\displaystyle\sum\limits_{1\leqslant j\leqslant\nu^{\prime}\left(  a^{\prime
}\right)  }}
c_{1,\sigma^{^{-1}}\left(  a^{\prime},j^{\prime}\right)  }=%
{\displaystyle\sum\limits_{1\leqslant m\leqslant\mu}}
\,%
{\displaystyle\sum\limits_{j^{\prime}\in J_{a^{\prime}}^{m}}}
c_{1,a_{m},\ell_{m}\left(  j^{\prime}\right)  }.
\]
As $\left(  c_{1}\right)  $ is generic for $X$, this imply that $\ell_{m}$ is
a bijection from $J_{a^{\prime}}^{m}$ onto $\left\{  1,...,\nu\left(
a_{m}\right)  \right\}  $. Hence, if $a$ is one of the $a_{m}$, any branch
$X_{a,j}$ of $X$ at $a$ is send by $\varphi$ to a branch of $X^{\prime}$ at
$a^{\prime}$. As any node of $X$ is send by $\varphi$ to a node of $X^{\prime
}$, this proves that $\varphi$ is actually continuous. If one of the $\left(
c_{\ell}^{\prime}\right)  $ is also generic for $X^{\prime}$, the same apply
for $\psi$.
\end{proof}

Now that reasonable uniqueness is achieved for the nodal IDN-problem, comes
the question of reconstructing solutions from the data boundary. The second
point of the result below shows how to first recover $F\left(  X\right)  $ and
$\partial\widetilde{u_{\ell}}$ from $\theta u_{\ell}$ and the intersection of
$F\left(  X\right)  $ with the lines $\Delta_{\xi}=\{z_{2}:=\frac{w_{2}}%
{w_{0}}=\xi\}$, $\xi\in\mathbb{C}$. Once this is done, the third point gives a
process to recover a normalization of $F\left(  X\right)  $. The fourth
statement enable to reconstruct $X$ itself if the admissible family is generic.

\begin{theorem}
\label{T/ looseIDN}Assume that $X$ is an open bordered nodal curve with
\textit{restricted DN-datum} $\left(  \gamma,u,\theta u\right)  $ associated
to admissible families without zero coefficient. Consider any normalization
$\widehat{X}\overset{\pi}{\longrightarrow}X$ of $X$. Then, the following holds

1) The map $\pi^{\ast}f$ where $f$ is defined by (\ref{F/ f}) has a
meromorphic extension $\widehat{F}$ to $X$ and there are discrete sets
$\widehat{T}$ and $S$ in $\widehat{X}$ and $Y=F\left(  X\right)  \backslash
f\left(  \gamma\right)  $ respectively such that $F:\widehat{X}\backslash
\widehat{T}\rightarrow Y\backslash S$ is one to one.

2) Almost all $\xi_{\ast}\in\mathbb{C}$ has a neighborhood $W_{\xi_{\ast}}$
such that for all $\xi$ in $W_{\xi_{\ast}}$, $Y_{\xi}=Y\cap\Delta_{\xi
}=\underset{1\leqslant j\leqslant p}{\cup}\{(h_{j}\left(  \xi\right)  ,\xi)\}$
where $h_{1},...,h_{p}$ are $p$ mutually distinct holomorphic functions on
$W_{\xi_{\ast}}$ whose symmetric functions $S_{h,m}=\underset{1\leqslant
j\leqslant p}{\Sigma}h_{j}^{m}$ can be recovered by the Cauchy type integral
formulas%
\begin{equation}
\frac{1}{2\pi i}\int_{\gamma}\frac{f_{1}^{m}}{f_{2}-\xi}df_{2}=S_{h,m}\left(
\xi\right)  +P_{m}\left(  \xi\right)  ,~m\in\mathbb{N}, \tag{%
$E_{m,\protect\xi}$%
}\label{F/ equation pour Y}%
\end{equation}
where $P_{m}$ is a polynomial of degree at most $m$. More precisely, the
system $E_{\xi}=(E_{m,\xi_{\nu}})_{\substack{0\leqslant m\leqslant M-1
\\o\leqslant\nu\leqslant N-1}}$ enables explicit computation of $h_{j}\left(
\xi_{\nu}\right)  $ and $P_{m}$ if $N\geqslant M\geqslant2p+1$ and $\xi
_{0},...,\xi_{N}$ are mutually distinct points.

3) Consider $\widehat{U_{\ell}}=\pi_{\ast}\widetilde{u_{\ell}}^{c_{\ell}}$,
$0\leqslant\ell\leqslant2$. Then $\partial\widehat{U_{\ell}}$ are meromorphic
$\left(  1,0\right)  $-forms and for almost all $\xi_{\ast}\in\mathbb{C}$,
$W_{\xi_{\ast}}$ can be chosen so that $S\cap\underset{\xi\in W_{\xi_{\ast}}%
}{\cup}Y_{\xi}=\varnothing$ and $\partial\widehat{U_{\ell}}$ can be
reconstructed in $\widehat{F}\,^{-1}(\underset{\xi\in W_{\xi_{\ast}}}{\cup
}Y_{\xi})$ from the well defined meromorphic quotient $(\partial
\widehat{U_{\ell}})/(\partial\widehat{F}_{2})$ thanks to the Cauchy type
formulas
\begin{equation}
\frac{1}{2\pi i}\int_{\gamma}\frac{f_{1}^{m}}{f_{2}-\xi}\theta u_{\ell}=%
{\displaystyle\sum\limits_{1\leqslant j\leqslant p}}
h_{j}\left(  \xi\right)  ^{m}\frac{\partial\widehat{U_{\ell}}}{\partial
\widehat{F}_{2}}\left(  \widehat{F}\,^{-1}\left(  h_{j}\left(  \xi\right)
,\xi\right)  \right)  \mathbb{+}Q_{m}\left(  \xi\right)  \tag{%
$T_{m,\protect\xi}$%
}\label{F/ equation Theta}%
\end{equation}
where $m$ is any integer and $Q$ is a polynomial of degree at most $m$.

4) Let $y$ be a singular point of $Y$ and $C$ a representative of an
irreducible component of the germ of $Y$ at $y$ whose boundary is a smooth
real curve of $\operatorname*{Reg}Y$. Then $C$ is the image by $F$ of a branch
of $X$ if and only if $\int_{\partial C}\partial F_{\ast}\widetilde{u_{\ell}%
}^{c_{\ell}}\neq0$. If so, $\frac{1}{2\pi i\,}\int_{\partial C}\partial
F_{\ast}\widetilde{u_{\ell}}^{c_{\ell}}=c_{a,j}$ where $a$ is a node of $X $
and $j$ is the index such that $X_{a,j}=F^{-1}\left(  C\right)  $ is one of
the branches of $X$ at $a$. Equivalently, $C$ is the image by $F$ of a branch
of $X$ if and only if $\int_{C}\left(  \partial F_{\ast}\widetilde{u_{\ell}%
}^{c_{\ell}}\right)  \wedge\overline{(\partial F_{\ast}\widetilde{u_{\ell}%
}^{c_{\ell}})}=+\infty$.
\end{theorem}

\noindent\textbf{Remarks. 1. }Let $p_{2}$ be the second natural projection of
$\mathbb{C}^{2}$ onto $\mathbb{C}$ and let $\gamma_{2}$ denote the real curve
$p_{2}\left(  f\left(  \gamma\right)  \right)  $. Then the points $\xi_{\ast}$
in the first statement can be any element of $\mathbb{C}\backslash\gamma_{2}$
outside a discrete set $\Delta$. More precisely, let $\Gamma$ be a connected
component of $\mathbb{C}\backslash\gamma_{2}$ and let $H$ be a holomorphic
function on $\mathbb{C}\times\Gamma$ such that $Y\cap p_{2}^{-1}\left(
\Gamma\right)  =\left\{  H=0\right\}  \cap p_{2}^{-1}\left(  \Gamma\right)  $
and $H_{z_{2}}=H\left(  .,z_{2}\right)  $ is a unitary polynomial for every
$z_{2}\in\Gamma$. Then, $\xi_{\ast}$ can be any element of the discrete set
where the discriminant $\operatorname{discr}H_{z_{2}}$ of $H_{z_{2}}$ don't
vanish. If $\xi_{\ast}$ is so, $W_{\xi_{\ast}}$ can be chosen to be $\left\{
\operatorname{discr}H_{z_{2}}\neq0\right\}  $. Note the equations $\left(
E_{m,\xi}\right)  $ enable to recover $H$.

\textbf{2. }The first three points of the theorem above are equivalent to
\cite[th.~2]{HeG-MiV2007}. (4) is a device for detecting whose singularities
of $Y$ are lying in $\pi{}^{-1}\left(  \operatorname*{Sing}X\right)  $ and
whose have appeared because $F$ is not necessarily an embedding. Note that the
characterizations given for $C$ has an invariant meaning. Indeed, if
$C=F\left(  B\right)  $ where $B$ is a branch of $X$ at $a$, then $\int
_{C}\left(  \partial F_{\ast}\widetilde{u_{\ell}}^{c_{\ell}}\right)
\wedge\overline{(\partial F_{\ast}\widetilde{u_{\ell}}^{c_{\ell}})}=\int
_{B}\left(  \partial\widetilde{u_{\ell}}^{c_{\ell}}\right)  \wedge
\overline{(\partial\widetilde{u_{\ell}}^{c_{\ell}})}$ and $\int_{\partial
C}\partial F_{\ast}\widetilde{u_{\ell}}^{c_{\ell}}=\int_{\partial B}%
\partial\widetilde{u_{\ell}}^{c_{\ell}}$.

\textbf{3.} If an admissible family has a zero coefficient, the above method
recovers branches corresponding to non zero coefficients but can't detect the
others. In particular, the method don't recognize nodes corresponding to zero
admissible family.

\textbf{4. }The fourth statement of the theorem enables to reconstruct in
$Z=\widehat{X}$, the set $S=\pi^{-1}\left(  \operatorname*{Sing}X\right)  $,
the map $\mu:S\longrightarrow\mathbb{N}$ such that $X_{\pi\left(  s\right)
,\mu\left(  s\right)  }=\pi\left(  W_{s}\right)  $ for some neighborhood
$W_{s}$ of $s$ in $Z$ and $c_{\ell,\pi\left(  s\right)  ,\mu\left(  s\right)
}=\frac{1}{2\pi i\,}\int_{\partial F\left(  W_{s}\right)  }\partial F_{\ast
}\widetilde{u_{\ell}}^{c_{\ell}}\overset{def}{=}\kappa_{\ell,s}$. If one of
the admissible families is known to be generic, then there is only one way to
split $S$ in non empty subsets $S_{1},..,S_{k}$ such that $X$ is the quotient
of $Z$ when for each $j$, points in $S_{j}$ are identified. This partition of
$S$ is also determined by the fact that if $T_{1},..,T_{\ell}$ is another
partition of $S$ such that $%
{\displaystyle\sum\limits_{s\in T_{m}}}
\kappa_{\ell,s}=0$ for any $m$, then each $T_{m}$ is the union of some the
$S_{1},...,S_{k}$. If no admissible family is known to be generic, $X$ is
roughly isomorphic to the nodal curve $X_{S}$ determined by $Z$ and
$S_{1},..,S_{k}$ and is obtained by supplementary identification\ in the set
of nodes of $X_{S}$.\medskip

\begin{proof}
[Proof]Denote by $F$ the meromorphic extension of $f$ to $X$ and put
$Y=F\left(  X\right)  \backslash\delta$ where $\delta=f\left(  \gamma\right)
$. Let $\widehat{X}\overset{\pi}{\longrightarrow}X$ be a normalization of $X$.
Then $\widehat{X}\overset{\widehat{F}=F\circ\pi}{\longrightarrow}Y$ is a
normalization of $Y$ and $\widehat{\gamma}=\pi^{-1}\left(  \gamma\right)  $ is
a real curve diffeomorphic to $\gamma$. The functions $\widehat{u_{\ell}}%
=\pi^{\ast}u_{\ell}$ are well defined functions on $\widehat{\gamma}$ and they
extends as harmonic functions $\widehat{U_{\ell}}$ on $\widehat{X}%
\backslash\pi^{-1}\left(  \operatorname*{Sing}X\right)  $ and as distributions
on $\widehat{X}$ with logarithmic singularities at points of $\pi^{-1}\left(
\operatorname*{Sing}X\right)  $. However, the forms $\partial\widehat{u_{\ell
}}$ extends to $\widehat{X}$ as meromorphic $\left(  1,0\right)  $-forms with
simple poles. Hence, $\widehat{f}=\pi^{\ast}f$ extends to $\widehat{X}$ as a
meromorphic function which is of course $\widehat{F}$. This is sufficient to
apply theorem~2 of \cite{HeG-MiV2007} whose statements are readily report here
by points (1) through (3). Notes that among the points of $\widehat{F}{}%
^{-1}\left(  \operatorname*{Sing}Y\right)  $ are those of $\pi{}^{-1}\left(
\operatorname*{Sing}X\right)  $ and the others which have appear only because
$F$ is not necessarily an embedding.

Since admissible families are assumed to be without zero coefficient, the
harmonic distributions $\widetilde{u_{\ell}}^{c_{\ell}}$ have logarithmic
singularities along each branch of at each node of $X$ while they are usual
harmonic function at regular points of $X$. This implies that the
$\widehat{U_{\ell}}=\pi_{\ast}\widetilde{u_{\ell}}^{c_{\ell}}$ have
singularities at each point above a node of $X$ and is bounded near points of
$\widehat{F}{}^{-1}\left(  \operatorname*{Sing}Y\right)  \backslash\pi{}%
^{-1}\left(  \operatorname*{Sing}X\right)  $. Since (3) explains how to
recover the $\partial\widehat{U_{\ell}}/\partial F_{2}$ and hence the
$\Theta_{\ell}=\partial\widehat{U_{\ell}}$, can be rebuilt from data boundary
by explicit formulas, the same applies for $\pi{}^{-1}\left(
\operatorname*{Sing}X\right)  $.

Let now $y\in Y$ be a singularity of $Y$ and $C$ an irreducible component of
$Y$ at $y$ whose boundary is a smooth real curve of $\operatorname*{Reg}Y$.
Assume that $C$ is the image by $F$ of a branch $X_{a,j}$ of $X$ where $a$ is
some node of $X$. Then $C$ is smooth and since when $z_{j}$ is a holomorphic
coordinate for $X_{a,j}$ centered at $a$, $\widetilde{u_{\ell}}^{c_{\ell}%
}\left\vert _{X_{a,j}^{\ast}}\right.  -2c_{a,j}\ln\left\vert z_{j}\right\vert
$ extends as a usual harmonic function on $X_{a,j}$, $\partial F_{\ast
}\widetilde{u_{\ell}}^{c_{\ell}}\left\vert _{X_{a,j}}\right.  $ is a
meromorphic $\left(  1,0\right)  $-form whose only singularity is at $y$ where
it has a simple pole with residue $c_{a,j}$. If $C$ contains no points of
$F\left(  \operatorname*{Sing}X\right)  $, then $\widetilde{u_{\ell}}%
^{c_{\ell}}$ is a usual harmonic function on $F^{-1}\left(  C\right)  $ and
$\int_{\partial C}\partial F_{\ast}\widetilde{u_{\ell}}^{c_{\ell}}$ has to be
zero. The equivalent characterization by the finiteness of the Dirichlet is straightforward.
\end{proof}

\subsubsection{Proof of theorem~\ref{T/ bordR3}}

The first part of the theorem is a particular case of theorem~\ref{T/ unicite}
applied to the nodal curves $\mathcal{X}$ and $\mathcal{X}^{\prime} $ obtained
by identifying in $X$ (resp. $X^{\prime}$) the points $a_{j}^{+}$ and
$a_{j}^{-}$ (resp. $a_{j}^{\prime+}$ and $a_{j}^{\prime-}$), $1\leqslant
j\leqslant\nu$. The reconstruction part is a particular case of
theorem~\ref{T/ looseIDN} applied to $\mathcal{X}$ and $\mathcal{X}^{\prime} $.

\section{Characterization of DN-data}

We give here some criterion to characterize what is a DN-datum. The theorem we
propose here is a development for the nodal type case of theorem~3 of
\cite{HeG-MiV2007}. Characterization theorems~3b and 4 of \cite{HeG-MiV2007}
and the characterization results of~\cite{WaR2008} also can be adapted to the
nodal type case but we avoid in this paper the heavy formulation they require.

For the sake of simplicity, open surfaces were assumed in the preceding
sections to have smooth boundaries even if they may be singular in their
interior. For a characterization device, we have to consider as in
\cite{HeG-MiV2007} almost smooth boundaries. An open bordered nodal curve $X$
has \textit{almost smooth boundary }$\gamma$ if$^{\text{(}}$\footnote{$h^{d}$
is the $d$-dimensionnial Hausdorff measure.}$^{\text{)}}$ $h^{2}\left(
\,X\cup\gamma\,\right)  <\infty$, $\gamma$ is a smooth oriented real curve
without component reduce to a point and if for some open neighborhood $W$ of
$\gamma$ in $X\cup\gamma$, $W\backslash\gamma$ is an open Riemann surface such
that the set $W_{\operatorname{sing}}$ of points of $\gamma$ where $W$ has not
smooth boundary satisfies $h^{1}\left(  W_{\operatorname{sing}}\right)  =0$.

If $X$ is an open bordered nodal curve with almost smooth boundary\textit{\ }%
$\gamma$, an adaptation of the preceding results to get as for classical
results (see e.g. \cite{AhL-SaL1960Livre}) that for any admissible family $c$,
a real valued function $u$ of class $C^{1}$ on $\gamma$ has a unique extension
$\widetilde{u}^{c}$ to $\overline{X}$ as a harmonic distribution such that
$\int_{W}i\,\partial\widetilde{u}^{c}\wedge\overline{\partial}\widetilde
{u}^{c}<+\infty$ for some open neighborhood $W$ of $\gamma$ in $\overline{X}$
and such that for each branch $X_{a,j}$ at a node $a$ of $X$,
$\operatorname*{Res}_{a}\left(  \left.  \partial u\right\vert _{X_{a,j}%
}\right)  =c_{a,j}$. Moreover, $N_{X,c}u$ still make sense as the element of
the dual space of $C^{1}\left(  \gamma\right)  $ which equals $\partial
\widetilde{u}^{c}/\partial\nu$ on $\gamma\backslash\overline{\mathcal{X}%
}_{\operatorname{sing}}$ (see \cite[prop.~12]{HeG-MiV2007}).\medskip

The first condition for $\left(  \gamma,u,\theta u\right)  $ to be a DN-datum
is for $\gamma$ to border a complex curve whose tangent bundle along $\gamma$
is the real two dimensional bundle given with the datum and to which $f$
extends meromorphically. Part (a) and (b) of the theorem beneath gives a
necessary and sufficient condition for this to occur. This fact does not
really depend of whether or not $u$ and $\theta u$ are restrictions of first
derivatives of a harmonic distribution. Part (c) gives a necessary and
sufficient condition for $u_{\ell}$ and $\theta u_{\ell}$ to be boundary
values coming from a harmonic distribution with logarithmic singularities. The
real curve $\gamma$ is assumed in part (b) and (c) to be connected for the
sake of simplicity.

\begin{theorem}
\label{T/ caract - a}Assume that hypothesis A is valid and consider$\vspace
{-3pt}$%
\begin{equation}
G:\mathbb{C}^{2}\ni\left(  \xi_{0},\xi_{1}\right)  \mapsto\frac{1}{2\pi i}%
\int_{\gamma}f_{1}\frac{d\left(  \xi_{0}+\xi_{1}f_{1}+f_{2}\right)  }{\xi
_{0}+\xi_{1}f_{1}+f_{2}}.\vspace{-3pt} \label{F/ fonction G}%
\end{equation}

(\textbf{a)} If an open bordered nodal curve $X$ has restricted DN-datum
$\left(  \gamma,u,\theta u\right)  $, then almost all point $\xi_{\ast}$\ of
$\mathbb{C}^{2}$\ has a neighborhood where one can find a find a
family,\ possibly empty, $\left(  h_{1},...,h_{p}\right)  $ of mutually
distinct holomorphic functions\ such that%
\begin{equation}
0=\frac{\partial^{2}}{\partial\xi_{0}^{2}}(G-%
{\displaystyle\sum\limits_{1\leqslant j\leqslant p}}
h_{j}) \label{F/ caract G}%
\end{equation}
and which satisfy the Riemann-Burgers equation%
\begin{equation}
h_{j}\frac{\partial h_{j}}{\partial\xi_{0}}=\frac{\partial h_{j}}{\partial
\xi_{1}},1\leqslant j\leqslant p. \label{F/ SW}%
\end{equation}

(\textbf{b)} Conversely, assume $\gamma$ is connected and the conclusion of
(a) is satisfied in a connected neighborhood $W_{\xi_{\ast}}$ of one point
$\left(  \xi_{0\ast},\xi_{1\ast}\right)  $.

Then, if $\left(  \partial^{2}G/\partial\xi_{0}^{2}\right)  _{\left\vert
W_{\xi_{\ast}}\right.  }\neq0$, there is an open Riemann surface $Z$ with
almost smooth boundary$^{\text{(}}$\footnote{With \cite[example 10.5]%
{HaR-LaB1975}, one can construct smooth restricted DN-datas for which the
solution of the IDN-problem is a manifold with only almost smooth
boundary.}$^{\text{)}}$ $\gamma$ where $f$\ extends meromorphically.

If $\left(  \partial^{2}G/\partial\xi_{0}^{2}\right)  _{\left\vert
W_{\xi_{\ast}}\right.  }=0$, the same conclusion holds for $\gamma$ or for
$-\gamma$ which denotes the same curve but endowed with the opposite
orientation. More precisely, when $\left(  \partial^{2}G/\partial\xi_{0}%
^{2}\right)  _{\left\vert W_{\xi_{\ast}}\right.  }=0$ and the expected
conclusion holds for $-\gamma$, $f\left(  -\gamma\right)  $ is the boundary
(in the sense of currents) of a possibly singular complex curve $Y$ of
$\mathbb{C}^{2}\backslash f\left(  \gamma\right)  $ and the expected
conclusion holds for $\gamma$ if and only if $Y$ is algebraic.

(\textbf{c)} Assume that $\left(  \,\overline{Z},\gamma\,\right)  $\ is a
Riemann surface with almost smooth boundary, let $\mathcal{D}$ be some smooth
domain in the double of $Z$ containing $\overline{Z}$ and let $g$ be a Green
function for $\mathcal{D}$.

Then, $\left(  \gamma,u,\theta u\right)  $\ is actually a restricted DN-datum
for the open bordered nodal curve $X$ obtained by identifying in $Z$ points
within each family $\left(  a_{n,j}\right)  _{1\leqslant j\leqslant\nu_{n}}$,
$1\leqslant n\leqslant N$, if and only if there exists a family of non zero
complex numbers $\left(  c_{\ell,n,j}\right)  _{\substack{0\leqslant
\ell\leqslant2 \\1\leqslant j\leqslant\nu_{n}}}$ such that $%
{\displaystyle\sum\limits_{1\leqslant j\leqslant\nu_{n}}}
c_{\ell,n,j}=0$ for any $\left(  \ell,n\right)  \in\left\{  0,1,2\right\}
\times\left\{  1,...,N\right\}  $ with the property that for any
$z\in\mathcal{D}\backslash\overline{Z}$ and any $\ell\in\left\{
0,1,2\right\}  $,%
\begin{equation}
\frac{2}{i}\int_{\gamma}u_{\ell}\left(  \zeta\right)  \partial_{\zeta}g\left(
\zeta,z\right)  +g\left(  \zeta,z\right)  \overline{\theta u_{\ell}\left(
\zeta\right)  }=2\pi%
{\displaystyle\sum\limits_{1\leqslant n\leqslant N}}
{\displaystyle\sum\limits_{1\leqslant j\leqslant\nu_{n}}}
c_{\ell,n,j}g\left(  a_{n,j},z\right)  \label{F/ Green}%
\end{equation}

\end{theorem}

\noindent\textbf{Remarks. 1.} (b) is actually the second part of
\cite[th.~3a]{HeG-MiV2007} with some precision about the case $\left(
\partial^{2}G/\partial\xi_{0}^{2}\right)  _{\left\vert W_{\xi_{\ast}}\right.
}=0$. The algebraic criterion is effective since $Y$ can be explicitly
reconstructed from the Cauchy type formulas given in theorem~\ref{T/ looseIDN}
or~\cite[th.~2]{HeG-MiV2007}

\textbf{2. }The $a_{n,j}$ and the $c_{\ell,n,j}$ are unique when they exist
because the right member of~(\ref{F/ Green}) extends to $Z$ as a distribution
$T$ such that $2i\partial\overline{\partial}T=2\pi%
{\displaystyle\sum\limits_{1\leqslant n\leqslant N}}
{\displaystyle\sum\limits_{1\leqslant j\leqslant\nu_{n}}}
c_{\ell,n,j}\delta_{a_{n,j}}dV$ where $dV$ is some volume form for $Z$.

\begin{proof}
Part (a) follows from~\cite[th.~3a]{HeG-MiV2007} since we can apply this
theorem to $Z$ where $Z\overset{\pi}{\longrightarrow}X$ is a normalization
because we only need to know that $\pi^{\ast}f$ embeds $\pi^{\ast}\gamma$ into
$\mathbb{CP}_{2}$ and extends meromorphically to $Z$. The case $\left(
\partial^{2}G/\partial\xi_{0}^{2}\right)  _{\left\vert W_{\xi_{\ast}}\right.
}\neq0$ of part~(b) is the same as~\cite[th.~3a,~B]{HeG-MiV2007}. Assume that
$\left(  \partial^{2}G/\partial\xi_{0}^{2}\right)  _{\left\vert W_{\xi_{\ast}%
}\right.  }=0$. As written in the proof of~\cite[th.~3a,~B]{HeG-MiV2007}, $\pm
f\left(  \gamma\right)  $ is then the boundary of a complex curve in some two
dimensional affine subspace of $\mathbb{CP}_{2}$ and therefor, the expected
conclusion holds for $\pm\gamma$. If both $-f\left(  \gamma\right)  $ and
$f\left(  \gamma\right)  $ are the boundaries in $\mathbb{CP}_{2}\backslash
f\left(  \gamma\right)  $ of a complex curve, namely $Y^{-}$ and $Y^{+}$, then
$Y^{-}\cup f\left(  \gamma\right)  \cup$ $Y^{+}$ is an algebraic complex curve.

Assume now that~(\ref{F/ Green}) is satisfied. Let $\Omega_{z}^{\ell}$ denote
the form $u_{\ell}\partial g\left(  .,z\right)  +g\left(  .,z\right)
\overline{\theta u_{\ell}}$ when $z$ is in $\mathcal{D}\backslash\gamma$ and
define $U_{\ell}^{+}$ (resp. $U_{\ell}^{-}$) the restriction to $\mathcal{D}%
^{+}=Z$ (resp. $\mathcal{D}^{-}=\mathcal{D}\backslash\overline{Z} $) of the
function $\mathcal{D}\backslash\gamma\ni z\mapsto\frac{2}{i}\int_{\gamma
}\Omega_{z}^{\ell}$. Then we know from the Plemelj-Sohotsky formula
of~\cite[lemma~15]{HeG-MiV2007} that $U_{\ell}^{\pm}$ is a harmonic function
on $\mathcal{D}^{\pm}$ which extends continuously to $\overline{\mathcal{D}%
^{\pm}}$, $u_{\ell}=U_{\ell}^{+}-U_{\ell}^{-}$ and $\theta u_{\ell}%
=\overline{\partial}U_{\ell}^{+}-\overline{\partial}U_{\ell}^{-}$ almost every
where on $\gamma$. Hence, $u_{\ell}$ is the boundary value of the distribution
$T_{\ell}$ on $Z$ defined by $T_{\ell}=U_{\ell}^{+}+2\pi%
{\displaystyle\sum\limits_{1\leqslant n\leqslant N}}
{\displaystyle\sum\limits_{1\leqslant j\leqslant\nu_{n}}}
c_{\ell,n,j}g\left(  a_{n,j},z\right)  $ and, by construction, $\partial
T_{\ell}=\theta u_{\ell} $ on $\gamma$.

If $X$ is the nodal curve obtained as in the theorem statement, then the
relations $%
{\displaystyle\sum\limits_{1\leqslant j\leqslant\nu_{n}}}
c_{\ell,n,j}=0$ implies that $T_{\ell}$ can be considered as a harmonic
distribution on $X$. As a consequence, $\left(  \gamma,u,\theta u\right)  $ is
a restricted DN-datum for $X$. The reciprocal follows directly from the definitions.
\end{proof}

\noindent\textbf{Acknowledgement}.\textit{\ The first author was partly
supported by FCP Kadry N 14.A18.21.0866 Russia.}

\renewcommand\baselinestretch{1}{\normalsize
\bibliographystyle{amsperso}
\bibliography{ref}

\providecommand{\bysame}{\leavevmode\hbox to3em{\hrulefill}\thinspace}
\providecommand{\MR}{\relax\ifhmode\unskip\space\fi MR }
\providecommand{\MRhref}[2]{%
  \href{http://www.ams.org/mathscinet-getitem?mr=#1}{#2}
}
\providecommand{\href}[2]{#2}
\begin{thebibliography}{10}

\bibitem{AhL-SaL1960Livre}
L.V. Ahlfors et L.~Sario, \emph{Riemann surfaces}, Princeton Mathematical
  Series, No. 26, Princeton University Press, Princeton, N.J., 1960.

\bibitem{BeM2003}
M.I. Belishev, \emph{The {C}alderon problem for two dimensional manifolds by
  the {BC}-method}, SIAM J. Math. Anal. \textbf{35} (2003), no.~1, 172--182.

\bibitem{BeM-KuY1992}
M.I. Belishev et Y.V. Kurylev, \emph{To the reconstruction of a {R}iemannian
  manifold via its spectral data ({BC}-method)}, Comm. Partial Differential
  Equations \textbf{17} (1992), no.~5-6, 767--804.

\bibitem{BuP1997}
P.~Buser, \emph{Inverse spectral geometry on {R}iemann surfaces}, Progress in
  inverse spectral geometry, Trends Math., Birkh\"auser, Basel, 1997,
  pp.~133--173.

\bibitem{SG2010Li}
H.P. de~Saint-Gervais, \emph{Uniformisation des surfaces de riemann}, Lyon, ENS
  Editions, 2010.

\bibitem{FeH1969Li}
H.~Federer, \emph{Geometric measure theory}, Die {G}rundlehren der
  mathematischen {W}issenschaften, Band 153, Springer-Verlag New York Inc.,
  New-York, 1969.

\bibitem{GaA1961}
A.M. Garsia, \emph{An imbedding of closed {R}iemann surfaces in {E}uclidean
  space}, Comment. Math. Helv. \textbf{35} (1961), 93--110.

\bibitem{GeI1962}
I.~M. Gelfand, \emph{Automorphic functions and the theory of representations},
  Proc. {I}nternat. {C}ongr. {M}athematicians ({S}tockholm, 1962), Inst.
  Mittag-Leffler, Djursholm, 1963, pp.~74--85.

\bibitem{GrA1958}
A.~Grothendieck, \emph{The cohomology theory of abstract algebraic varieties},
  Proc. {I}nternat. {C}ongress {M}ath. ({E}dinburgh, 1958), Cambridge Univ.
  Press, New York, 1960, pp.~103--118.

\bibitem{Ha-Mo1998Li}
J.~Harris et I.~Morrison, \emph{Moduli of curves}, Graduate Texts in
  Mathematics, vol. 187, Springer-Verlag, New York, 1998.

\bibitem{HaR166Li}
R.~Hartshorne, \emph{Residues and duality}, Lecture notes of a seminar on the
  work of A. Grothendieck, given at Harvard 1963/64. With an appendix by P.
  Deligne. Lecture Notes in Mathematics, No. 20, Springer-Verlag, Berlin, 1966.

\bibitem{HaR-LaB1975}
F.~R. Harvey et H.~B. Lawson, \emph{Boundaries of complex analytic varieties},
  Ann. of Math. \textbf{102} (1975), 233--290.

\bibitem{HaR1977}
R.~Harvey, \emph{Holomorphic chains and their boundaries}, Proc. Symp. Pure
  Math. \textbf{30} (1977), 309--382.

\bibitem{HeG-MiV2012}
G.~Henkin et V.~Michel, \emph{Problème de {P}lateau complexe feuilleté.
  phénomènes de {H}artogs-{S}everi et {B}ochner pour des feuilletages {CR}
  singuliers}, http://arxiv.org/abs/1109.4300, 2012.

\bibitem{HeG-MiV2007}
G.M. Henkin et V.~Michel, \emph{On the explicit reconstruction of a {R}iemann
  surface from its {D}irichlet-{N}eumann operator}, Geom. Funct. Anal.
  \textbf{17} (2007), no.~1, 116--155.

\bibitem{LaM-UhG2001}
M.~Lassas et G.~Uhlmann, \emph{On determining a {R}iemannian manifold from the
  {D}irichlet-to-{N}eumann map}, Ann. Sci. \'Ecole Norm. Sup. (4) \textbf{34}
  (2001), no.~5, 771--787.

\bibitem{RoM1954}
M.~Rosenlicht, \emph{Generalized {J}acobian varieties}, Ann. of Math. (2)
  \textbf{59} (1954), 505--530.

\bibitem{RuR1971}
R.A. R{\"u}edy, \emph{Embeddings of open {R}iemann surfaces}, Comment. Math.
  Helv. \textbf{46} (1971), 214--225.

\bibitem{SeJ1959}
J.-P. Serre, \emph{Groupes alg\'ebriques et corps de classes}, Publications de
  l'institut de math\'ematique de l'universit\'e de Nancago, VII. Hermann,
  Paris, 1959.

\bibitem{SyJ1990}
J.~Sylvester, \emph{An anisotropic inverse boundary value problem}, Comm. Pure
  Appl. Math. \textbf{43} (1990), 201--232.

\bibitem{WaR2008}
R.A. Walker, \emph{Extended shockwave decomposability related to boundaries of
  holomorphic 1-chains within {$\Bbb C\Bbb P^2$}}, Indiana Univ. Math. J.
  \textbf{57} (2008), no.~3, 1133--1172.

\bibitem{WiK1966}
K.-W. Wiegmann, \emph{Einbettungen komplexer {R}\"aume in {Z}ahlenr\"aume},
  Invent. Math. \textbf{1} (1966), 229--242.

\end{thebibliography}
}%

\end{spacing}%

\end{document}